\documentclass{article}

\usepackage{microtype}
\usepackage{graphicx}
\usepackage[tight]{subfigure}
\usepackage{booktabs} 
\usepackage{amsmath,amsthm,amssymb,amsfonts}
\usepackage{mathtools}
\usepackage{thmtools,thm-restate}

\usepackage[extdef=true]{delimset}
\usepackage{xspace}

\usepackage[round]{natbib}
\usepackage{algorithm}
\usepackage{algorithmic}
\usepackage{authblk}

\usepackage{enumitem}
\usepackage[margin=1in]{geometry}

\usepackage[colorlinks=true,allcolors=blue]{hyperref}

\usepackage[capitalize]{cleveref}
\usepackage{crossreftools}

\declaretheorem[parent=]{theorem}
\declaretheorem[sibling=theorem]{lemma}
\declaretheorem[sibling=theorem]{corollary}

\DeclareMathOperator*{\argmin}{arg\,min}

\allowdisplaybreaks

\graphicspath{ {images/} }

\newcommand{\alglinelabel}{%
  \addtocounter{ALC@line}{-1}
  \refstepcounter{ALC@line}
  \label
}
\newcommand{\supplementaryTitle}{%
\section*{Appendix}
}
\newcommand{\supplementaryMaterial}{appendix\xspace}

\crefname{ALC@line}{line}{lines}

\begin{document}

\title{Asynchronous Stochastic Optimization \\ Robust to Arbitrary Delays}

\author[1]{Alon Cohen}
\author[1,2]{Amit Daniely}
\author[1]{Yoel Drori}
\author[1,3]{Tomer Koren}
\author[1]{Mariano Schain}
\affil[1]{Google Research Tel Aviv}
\affil[2]{Hebrew University of Jerusalem}
\affil[3]{Blavatnik School of Computer Science, Tel Aviv University}

\maketitle

\begin{abstract}
  We consider stochastic optimization with delayed gradients where, at each time step~$t$, the algorithm makes an update using a stale stochastic gradient from step $t - d_t$ for some arbitrary delay $d_t$.
  This setting abstracts asynchronous distributed optimization where a central server receives gradient updates computed by worker machines. These machines can experience computation and communication loads that might vary significantly over time.
  In the general non-convex smooth optimization setting, we give a simple and efficient algorithm that requires $O( \sigma^2/\epsilon^4 + \tau/\epsilon^2 )$ steps for finding an $\epsilon$-stationary point $x$, where $\tau$ is the \emph{average} delay $\smash{\frac{1}{T}\sum_{t=1}^T d_t}$ and $\sigma^2$ is the variance of the stochastic gradients.
  This improves over previous work, which showed that stochastic gradient decent achieves the same rate but with respect to the \emph{maximal} delay $\max_{t} d_t$, that can be significantly larger than the average delay especially in heterogeneous distributed systems.
  Our experiments demonstrate the efficacy and robustness of our algorithm in cases where the delay distribution is skewed or heavy-tailed.
\end{abstract}

\newcommand{\E}{\mathbb{E}}
\newcommand{\ind}{I}
\newcommand{\reals}{{\mathbb{R}}}
\newcommand{\dotp}{\boldsymbol{\cdot}}
\newcommand{\ifrac}[2]{#1/#2}
\newcommand{\ipfrac}[2]{(#1)/(#2)}

\newcommand{\algo}{Picky SGD\xspace}
\newcommand{\simgenworker}{Generate Async Sequence - Worker\xspace}
\newcommand{\simgenmaster}{Generate Async Sequence\xspace}
\newcommand{\simalg}{Simulate Async\xspace}
\newcommand{\ac}[1]{\textcolor{red}{\bf \{AC: #1\}}}
\newcommand{\tk}[1]{\textcolor{magenta}{\bf \{TK: #1\}}}
\newcommand{\mrs}[1]{\textcolor{blue}{\bf \{MS: #1\}}}
\newcommand{\yd}[1]{\textcolor{cyan}{\bf \{YD: #1\}}}

\section{Introduction}

Gradient-based iterative optimization methods are widely used in large-scale
machine learning applications as they are extremely simple to implement and use,
and come with mild computational requirements. 
On the other hand, in their standard formulation they are also inherently serial and synchronous due to their iterative nature. 
For example, in stochastic gradient descent (SGD), each step involves an update
of the form $x_{t+1} = x_t - \eta g_t$ where $x_t$ is the current iterate, and
$g_t$ is a (stochastic) gradient vector evaluated at $x_t$. To progress to the
next step of the method, the subsequent iterate $x_{t+1}$ has to be fully
determined by the end of step $t$ as it is required for future gradient queries.
Evidently, this scheme has to wait for the computation of the gradient $g_t$ to
complete (this is often the most computationally intensive part in SGD) before
it can evaluate $x_{t+1}$.

In modern large scale machine learning applications, a direct serial
implementation of gradient methods like SGD is overly costly, and parallelizing
the optimization process over several cores or machines is desired.
Perhaps the most common parallelization approach is via \emph{mini-batching},
where computation of stochastic gradients is distributed across several worker
machines that send updates to a parameter server. The parameter server is responsible for
accruing the individual updates into a single averaged gradient, and
consequently, updating the optimization parameters using this gradient.

While mini-batching is well understood
theoretically~\citep[e.g.,][]{lan2012optimal,dekel2012optimal,cotter2011better,duchi2012randomized},
it is still fundamentally synchronous in nature and its performance is adversely
determined by the slowest worker machine: the parameter server must wait for all
updates from all workers to arrive before it can update the model it maintains.
This could cause serious performance issues in heterogeneous distributed
networks, where worker machines may be subject to unpredictable loads that vary
significantly between workers (due to different hardware, communication
bandwidth, etc.) and over time (due to varying users load, power outages, etc.).

An alternative approach that has recently gained popularity is to employ
\emph{asynchronous} gradient
updates~\citep[e.g.,][]{nedic2001distributed,agarwal2012distributed,chaturapruek2015asynchronous,lian2015asynchronous,feyzmahdavian2016asynchronous};
namely, each worker machine computes gradients independently of the other
machines, possibly on different iterates, and sends updates to the parameter
server in an asynchronous fashion. This implies the parameter server might be
making \emph{stale updates} based on \emph{delayed} gradients taken at earlier,
out-of-date iterates. While these methods often work well in practice, they have
proven to be much more intricate and challenging to analyze theoretically than
synchronous gradient methods, and overall our understanding
of asynchronous updates remains lacking.


Recently, \citet{arjevani2020tight} and subsequently \citet{stich2020error} have
made significant progress in analyzing delayed asynchronous gradient methods.
They have shown that in stochastic optimization, delays only affect a
lower-order term in the convergence bounds.
In other words, if the delays are not too large, the convergence rate of SGD may not be affected by the delays. 
(\citealp{arjevani2020tight} first proved this for quadratic objectives; \citealp{stich2020error} then proved a more general result for smooth functions.)
More concretely, \citet{stich2020error} showed that SGD with a sufficiently attenuated step size to account for the delays attains an iteration complexity bound of the form  
\begin{align} \label{eq:stich}
    O\brk3{\frac{\sigma^2}{\epsilon^4} + \frac{\tau_{\max}}{\epsilon^2}}
\end{align}
for finding an $\epsilon$-stationary point of a possibly non-convex smooth
objective function (namely, a point at which the gradient is of norm $\leq
\epsilon$). Here $\sigma^2$ is the variance of the noise in the stochastic
gradients, and $\tau_{\max}$ is the \emph{maximal} possible delay, which
is also needed to be known a-priori for properly tuning the SGD step size.
Up to the $\tau_{\max}$ factor in the second term, this bound is identical
to standard iteration bounds for stochastic non-convex SGD without
delays~\citep{ghadimi2013stochastic}.

While the bound in \cref{eq:stich} is a significant improvement over previous
art, it is still lacking in one important aspect: the dependence on the maximal
delay could be excessively large in truly asynchronous environments, making the second term in the bound the dominant term.
For example, in heterogeneous or massively distributed networks, the maximal delay is effectively determined by the single slowest (or less reliable) worker machine---which is precisely the issue with synchronous methods we set to address in the first place. 
Moreover, as \citet{stich2020error} show, the step size used to achieve the bound in \cref{eq:stich} could be as much as $\tau_{\max}$-times smaller than that of without delays, which could severely
impact performance in practice.

\subsection{Contribution}

We propose a new algorithm for stochastic optimization with asynchronous delayed updates, we call ``\algo,''
that is significantly more robust than SGD, especially when the (empirical)
distribution of delays is skewed or heavy-tailed and thus the maximal
delay could be very large.
For general smooth possibly non-convex objectives, our algorithm achieves a convergence bound of the form
$$
    O\brk3{\frac{\sigma^2}{\epsilon^4} + \frac{\tau_\mathrm{avg}}{\epsilon^2}},
$$
where now $\tau_\mathrm{avg}$ is the \emph{average} delay in retrospect. 
This is a significant improvement over the bound in \cref{eq:stich} whenever
$\tau_\mathrm{avg} \ll \tau_\mathrm{max}$, which is indeed the case with
heavy-tailed delay distributions. 
Moreover, \algo is very efficient, extremely simple to implement, and does \emph{not} require to know the average delay $\tau_\mathrm{avg}$ ahead of time for optimal tuning. 
In fact, the algorithm only relies on a single additional hyper-parameter beyond the step-size.
 
Notably, and in contrast to SGD as analyzed in previous work~\citep{stich2020error}, our
algorithm is able to employ a significantly larger effective step size,
and thus one could expect it to perform well in practice compared to SGD.
Indeed, we show in experiments that \algo is able to converge quickly on
large image classification tasks with a relatively high learning rate, even when
very large delays are introduced. In contrast, in the same setting, SGD needs to
be configured with a substantially reduced step size to be able to converge at all,
consequently performing poorly compared to our algorithm.

Finally, we also address the case where $f$ is smooth and convex, in which we give a close variant of our algorithm with an iteration complexity bound of the form
$$
    O\brk3{\frac{\sigma^2}{\epsilon^2} + \frac{\tau_\mathrm{avg}}{\epsilon}}
$$
for obtaining a point $x$ with $f(x)-f(x^*) \leq \epsilon$ (where $x^*$ is a minimizer of $f$ over $\reals^d$).
Here as well, our rate matches precisely the one obtained by the state-of-the-art~\citep{stich2020error}, but with the dependence on the maximal delay being replaced with the average delay.
For consistency of presentation, we defer details on the convex case to \cref{sec:convex} and focus here on our algorithm for non-convex optimization.



Concurrently to this work,~\citet{pmlr-v139-aviv21a} derived similar bounds that depend on the average delay. Compared to our contribution, their results are adaptive to the smoothness and noise parameters, but on the other hand, are restricted to convex functions and their algorithms are more elaborate and their implementation is more involved.

\subsection{Additional related work}

For general background on distributed asynchronous optimization and basic asymptotic convergence results, we refer to the classic book by~\citet{bertsekas1997parallel}.
%
%
Since the influential work of \citet{niu2011hogwild}, there has been 
significant interest in asynchronous algorithms in a related model where there
is a delay in updating individual \emph{parameters} in a shared parameter vector
(e.g., \citep{reddi2015variance,mania2017perturbed,zhou2018simple,leblond2018improved}). This is of course very
different from our model, where steps use the full gradient vector in atomic, yet delayed, updates.


Also related to our study is the literature on Local SGD (e.g., \citealp{woodworth2020local} and references therein), which is a distributed gradient method that perform several local (serial) gradient update steps before communicating with the parameter server or with other machines. Local SGD methods have become popular recently since they are used extensively in Federated Learning~\citep{mcmahan2017communication}.
%
%
We note that the theoretical study in this line of work is mostly concerned with analyzing \emph{existing} 
distributed variants of SGD used in practice, whereas we aim to develop and
analyze \emph{new} algorithmic tools to help with mitigating the effect of stale
gradients in asynchronous optimization.


A related yet orthogonal issue in distribution optimization, which we do not address here, is reducing the communication load between the workers and servers. One approach that was recently studied extensively is doing this by compressing gradient updates before they are transmitted over the network. We refer to \cite{alistarh2017qsgd,karimireddy2019error,stich2020error} for further discussion and references.



\section{Setup and Basic Definitions} \label{sec:prelims}

\subsection{Stochastic non-convex smooth optimization}

We consider stochastic optimization of a $\beta$-smooth (not necessarily convex) non-negative function $f$ defined over the $d$-dimensional Euclidean space $\reals^d$.
A function $f$ is said to be $\beta$-smooth if it is differentiable and its gradient operator is $\beta$-Lipschitz, that is, if $\norm{\nabla f(x) - \nabla f(y)} \leq \beta \norm{x-y}$ for all $x,y \in \reals^d$. This in particular implies (e.g., \citep{nesterov2003introductory}) that for all $x,y \in \reals^d$, 
\begin{align} \label{eq:smoothness}
    f(y)
    \leq
    f(x) + \nabla f(x) \dotp (y-x) + \frac{\beta}{2} \norm{y-x}^2
    .
\end{align}
We assume a stochastic first-order oracle access to $f$; namely, $f$ is endowed
with a stochastic gradient oracle that given a point $x \in \reals^d$ returns a
random vector $\tilde g(x)$, independent of all past randomization, such that
$\E[\tilde g(x) \mid x] = \nabla f(x)$ and $\E[\norm{\tilde g(x) - \nabla
f(x)}^2 \mid x] \leq \sigma^2$ for some variance bound $\sigma^2 \geq 0$.
In this setting, our goal is to find an $\epsilon$-stationary point of~$f$, namely, a point $x\in \reals^d$ such that $\| \nabla f(x)\| \le \epsilon$, with as few samples of stochastic gradients as possible.

\subsection{Asynchronous delay model}

We consider an abstract setting where stochastic gradients (namely, outputs for invocations of the stochastic first-order oracle)
are received asynchronously and are subject to arbitrary delays.
The asynchronous model can be abstracted as follows.
We assume that at each step $t$ of the optimization, the algorithm obtains a
pair $(x_{t-d_t},g_t)$ where $g_t$ is a stochastic gradient at $x_{t-d_t}$ with
variance bounded by $\sigma^2$; namely, $g_t$ is a random vector such that $\E_t
g_t = \nabla f(x_{t-d_t})$ and $\E_t\norm{g_t-\nabla f(x_{t-d_t})}^2 \le
\sigma^2$ for some delay $0\le d_t < t$. Here and throughout, $\E_t[\cdot]$
denotes the expectation conditioned on all randomness drawn before step $t$.
After processing the received gradient update, the algorithm may query a new
stochastic gradient at whatever point it chooses (the result of this query will
be received with a delay, as above).

Few remarks are in order:

\begin{itemize}
\item 
We stress that the delays $d_1,d_2,\ldots$ are entirely arbitrary, possibly
chosen by an adversary; in particular, we do \emph{not} assume they are sampled
from a fixed stationary distribution. Nevertheless, we assume that the
delays are independent of the randomness of the stochastic gradients (and of the
internal randomness of the optimization algorithm, if any).%
\footnote{One can thus think of the sequence of delays as being fixed ahead of
time by an oblivious adversary.}
\item 
For simplicity, we assumed above that a stochastic gradient is received at every round $t$. This is almost without loss of generality:%
\footnote{We may, in principle, allow to query the stochastic gradient oracle even on rounds where no feedback is received, however this would be redundant in most reasonable instantiations of this model~(e.g., in a parameter server architecture).}
if at some round no feedback is observed, we may simply skip the round without affecting the rest of the optimization process (up to a re-indexing of the remaining rounds).
%
\item
Similarly, we will also assume that only a single gradient is obtained in each
step; the scenario that multiple gradients arrive at the same step (as in
mini-batched methods) can be simulated by several subsequent iterations in each
of which a single gradient is processed.
\end{itemize}

\section{The \algo Algorithm}

We are now ready to present our asynchronous stochastic optimization algorithm, which we call \algo; see pseudo-code in \cref{alg:sgd-with-delays}.
The algorithm is essentially a variant of stochastic gradient descent, parameterized by a learning rate $\eta$ as well as a target accuracy $\epsilon$.

\begin{algorithm}[ht]
    \caption{\algo} \label{alg:sgd-with-delays}
    \begin{algorithmic}[1]
        \STATE {\bf input}: 
            learning rate $\eta$, 
            target accuracy $\epsilon$.
        \FOR{$t=1,\ldots,T$}
            \STATE {\bf receive} delayed stochastic gradient $g_t$ and point $x_{t-d_t}$ such that $\E_t[g_t] = \nabla f(x_{t-d_t})$. \alglinelabel{ln:receive-grad}
            \IF{$\norm{x_t - x_{t-d_t}} \le \ifrac{\epsilon}{(2 \beta)}$}\alglinelabel{ln:compare-dist} \alglinelabel{ln:test-update}
                \STATE {\bf update:} $x_{t+1} = x_t - \eta g_t$.
            \ELSE 
                \STATE {\bf pass:} $x_{t+1} = x_t$.
            \ENDIF
        \ENDFOR
    \end{algorithmic}
\end{algorithm}

\algo maintains a sequence of iterates $x_1,\ldots,x_T$.
At step $t$, the algorithm receives a delayed stochastic gradient $g_t$ that was computed at an earlier iterate $x_{t-d_t}$ (\cref{ln:receive-grad}). 
Then, in \cref{ln:test-update}, the algorithm tests whether $\norm{x_t - x_{t-d_t}} \le \ifrac{\epsilon}{2 \beta}$. 
Intuitively, this aims to verify whether the delayed (expected) gradient $\nabla f(x_{t-d_t})$ is ``similar'' to the gradient $\nabla f(x_t)$ at the current iterate $x_t$;  due to the smoothness of $f$, we expect that if $x_{t-d_t}$ is close to $x_t$, then also the corresponding gradients will be similar. If this condition holds true, the algorithm takes a gradient step using $g_t$ with step size $\eta$.

Our main theoretical result is the following guarantee on the success of the algorithm.

\begin{theorem} \label{thm:main}
    Suppose that \cref{alg:sgd-with-delays} is initialized at $x_1 \in \reals^d$ with $f(x_1) \leq F$ and ran with
    \[
        T 
        \geq 
        500 \beta F \brk3{\frac{\sigma^2}{\epsilon^4} + \frac{\tau+1}{\epsilon^2}}
        ,\quad
        \eta 
        = 
        \frac{1}{4 \beta} \min \brk[c]3{1, \frac{\epsilon^2}{\sigma^2}}
        ,
    \]
    where $\tau$ be the average delay, i.e., $\tau = (\ifrac{1}{T}) \sum_{t=1}^T d_t$.
    Then, with probability at least $\frac12$, there is some $1 \leq t \leq T$ for which $\norm{\nabla f(x_t)} \le \epsilon$.
\end{theorem}

Observe that the optimal step size in \cref{thm:main} is independent of the average delay $\tau$. This is important for two main reasons: (i) implementing
the algorithm does not require knowledge about future, yet-to-be-seen delays; and
(ii) even with very large delays, the algorithm can maintain a high effective step size.

We note that the guarantee of \cref{thm:main} is slightly different from typical bounds in non-convex optimization (e.g., the bounds appearing in the previous work \cite{karimireddy2019error}): our result claims about the \emph{minimal} gradient norm of any iterate rather than the \emph{average} gradient norm over the iterates. 
Arguably, this difference does not represent a very strong limitation: the significance of convergence bounds in non-convex optimization is, in fact, in that they ensure that one of the iterates along the trajectory of the algorithm is indeed an approximate critical point, and the type of bound we establish is indeed sufficient to ensure exactly that.

We further note that while the theorem above only guarantees a constant success probability, it is not hard to amplify this probability to an arbitrary $1-\delta$ simply by restarting the algorithm $O(\log(1/\delta))$ times (with independent stochastic gradients); with high probability, one of the repetitions will be successful and run through a point with gradient norm~$\leq \epsilon$, which would imply the guarantee in the theorem with probability at least $1-\delta$.

\section{Analysis}

In this section we analyze \cref{alg:sgd-with-delays} and prove our main result.
Throughout, we denote $x'_t = x_{t-d_t}$ and let $N_t$ denote the noise vector at step $t$, namely $N_t = g_t - \nabla f(x'_t)$. Note that $\E\brk[s]{N_t \mid x_t,x_t'} = 0$ and $\E\brk[s]{\norm{N_t}^2 \mid x_t,x_t'} \leq \sigma^2$, since the iterates $x_t,x_t'$ are conditionally independent of the noise in $g_t$ as this gradient is obtained by the algorithm only at step $t$, after $x_t,x_t'$ were determined.

To prove~\cref{thm:main}, we will analyze a variant of the algorithm that will
stop making updates once it finds a point with $\norm{\nabla f(x)} \le
\epsilon$ (and eventually {\em fails} otherwise). That is, if $\norm{x_t-x'_t} > \ifrac{\epsilon}{2 \beta}$ or
$\norm{\nabla f(x_t)} \le \epsilon$ then $x_{t+1} = x_t$. Else, $x_{t+1} = x_t -
\eta g_t$. This variant is impossible to implement (since it needs to compute
the exact gradient at each step), but the guarantee of \cref{thm:main} is valid
for this variant if and only if it is valid for the original algorithm: one
encounters an $\epsilon$-stationary point if and only if the other does so.


First, we prove a simple technical lemma guaranteeing that whenever the algorithm takes a step, a large gradient norm implies a large decrease in function value. It is a variant of the classical ``descent lemma,'' adapted to the case where the gradient step is taken with respect to a gradient computed at a nearby point.

\begin{lemma} \label{lem:loacl_improvement}
    Fix $x,x'\in \reals^d$ with $\norm{x-x'} \le \ifrac{\epsilon}{2 \beta}$ and $\norm{\nabla f(x') } > \epsilon$. 
    Let $N\in \reals^d$ be a random vector with $\E\brk[s]{N \mid x,x'} = 0$ and $\E\brk[s]{\norm{N}^2 \mid x,x'} \le \sigma^2$.
    Then,
    \begin{align*}
        \E\brk[s]{f \brk{x - \eta \brk{\nabla f(x') + N}}} - \E f(x)
        \le 
        -\frac{\eta}{2} \E\norm{\nabla f(x')}^2
        + 
        \frac{\eta^2 \beta}{2} \brk{\sigma^2 + \E\norm{\nabla f(x')}^2}
        .
    \end{align*}
    In particular, for our choice of $\eta$, we have
    \begin{align} \label{eq:descent}
        \frac{\eta}{4} \E\norm{\nabla f(x')}^2 
        \le
        \E f(x) - \E\brk[s]{f \brk1{x - \eta \brk{\nabla f(x') + N}}}
        .
    \end{align}
\end{lemma}

\begin{proof}
    Using the smoothness of $f$ (\cref{eq:smoothness}), we have
    \begin{align*}
        f(x - \eta \brk{\nabla f(x') + N}) - f(x)
        \leq
        -\eta \nabla f(x) \dotp \brk{\nabla f(x') + N} 
        + 
        \tfrac12 \eta^2\beta \norm{\nabla f(x') + N}^2
        .
    \end{align*}
    Taking expectation over $N$ conditioned on $x,x'$, we get
    \begin{align*}
        &\mkern-36mu
        \E\brk[s]{f(x - \eta \brk{\nabla f(x') + N}) - f(x) \mid x,x'}
        \\
        &\leq
        -\eta \nabla f(x) \dotp \nabla f(x')   
        + 
        \tfrac12 \eta^2\beta \brk{\|\nabla f(x')\|^2 + \sigma^2}
        \\
        &=
        -\eta \nabla f(x')   \dotp \nabla f(x')  
        -
        \eta \nabla f(x') \dotp (\nabla f(x) - \nabla f(x'))
        + 
        \tfrac12 \eta^2\beta \brk{\norm{\nabla f(x')}^2 + \sigma^2}
        \\
        &\leq
        -\eta \|\nabla f(x')\|^2 
        +  
        \eta \beta \|\nabla f(x')\| \|x-x'\|
        + 
        \tfrac12 \eta^2\beta \brk{\|\nabla f(x')\|^2 + \sigma^2}
        \\
        &= 
        \eta \brk{\beta \|\nabla f(x')\| \|x-x'\|  
        - \|\nabla f(x')\|^2}
        + 
        \tfrac12 \eta^2\beta \brk{\|\nabla f(x')\|^2 + \sigma^2}
        .
    \end{align*}
    Since $\epsilon \le \| \nabla f(x')\| $ then 
    $$
        \|x-x'\| \leq \frac{\epsilon}{2\beta} 
        \leq 
        \frac{1}{2\beta} \norm{\nabla f(x')},
    $$ 
    and we have
    \begin{align*}
        \E\brk[s]!{f(x - \eta \brk{\nabla f(x') + N}) - f(x) \mid x,x'}
        \le 
        -\frac{\eta}{2} \|\nabla f(x')\|^2
        + 
        \tfrac12 \eta^2\beta \brk{\sigma^2 + \|\nabla f(x')\|^2}
        .
    \end{align*}
    If $\epsilon \ge  \sigma$ then $\sigma^2 \le \norm{\nabla f(x')}^2$. This, with $\eta = \ifrac{1}{4\beta}$, yields \cref{eq:descent}.
    If $\epsilon < \sigma$ and $\eta = \ifrac{\epsilon^2}{4 \sigma^2 \beta}$, then $\eta^2 \le \ifrac{\epsilon^2}{16\sigma^2 \beta^2}$. Plugging that in instead, using $\norm{\nabla f(x')} \ge \epsilon$, and taking expectations (with respect to $x,x'$) gets us \cref{eq:descent}.
\end{proof}

We next introduce a bit of additional notation. 
We denote by $\ind_t$ the indicator of event that the algorithm performed an update at time $t$. 
Namely, 
$$
    I_t 
    = 
    \ind\set!{ \| x_t-x'_t \| \le \ifrac{\epsilon}{2\beta} ~\;\text{and}\; \|\nabla f(x_t)\| > \epsilon}
    .
$$
Note that $\ind_t=1$ implies that $\|\nabla f(x_s)\| \ge \epsilon$ for all $s =
1,\ldots,t$. Further, we denote by $\Delta_t = f(x_{t}) - f(x_{t+1})$ the
improvement at time $t$. Since $f$ is non-negative and $f(x_1)\le F$, we have
that for all $t$, $$ \sum_{i=1}^t \Delta_i = f(x_1) - f(x_{t+1}) 
    \leq 
    F
    .
$$
Note that by \cref{lem:loacl_improvement} we have that $\E\Delta_t \ge 0$. 
The rest of the proof is split into two cases: $\sigma\le \epsilon$, and $\sigma \geq \epsilon$.

\subsection{Case (i): \texorpdfstring{$\boldsymbol{\sigma\le \epsilon}$}{}}

This regime is intuitively the ``low noise'' regime in which the standard deviation of the gradient noise, $\sigma$, is smaller than the desired accuracy $\epsilon$. We prove the following.

\begin{lemma} \label{lem:no-noise}
    Suppose that $\sigma\le \epsilon$ and the algorithm fails with probability $\geq \tfrac12$. 
    Then $T \le \ifrac{128 \beta F (\tau +1)}{\epsilon^2}$.
\end{lemma}

To prove the lemma above, we first show that the algorithm must make a significant number of updates, as shown by the following lemma.

\begin{lemma} \label{lem:update-lower-bound}
    If the algorithm fails, then the number of updates that it makes is at least $T/4(\tau + 1)$.
\end{lemma}

\begin{proof}
    Consider $U_{2\tau}$, the number of steps $t$ for which the delay $d_t$ is at least $2\tau$. 
    We must have $U_{2\tau} \le T/2$ (otherwise the total sum of delays exceeds $\tau T$, contradicting the definition of $\tau$). On the other hand,
    let $k$ be the number of updates that the algorithm makes.
    Let $t_1 < t_2 < ... < t_k$ be the steps in which an update is made. Denote $t_0 = 0$ and $t_{k+1} = T$.
    Now, fix $i$ and consider the steps at times $s_n = t_i + n$ for
    $n \in [1, 2, \ldots , t_{i+1} - t_i - 1]$. 
    In all those steps no update takes place and $x_{s_n} = x_{t_i}$. We must have $d_{s_n} > n$ for all $n$ (otherwise $x_t = x_{t-d_t}$ for $t=s_n$ and an update occurs). In particular we have that $d_{s_n} \ge 2\tau$ in at least $t_{i+1} - t_{i} - 1 - 2\tau$ steps in $[t_{i},t_{i+1}]$. Hence, 
    \[
       U_{2\tau} \ge \sum_{i=0}^{k-1} \brk{t_{i+1} - t_{i} - 1 - 2\tau}  = T - k(1 + 2\tau).
    \]
    %
    %
    Finally, it follows that
    $
        T - k(1 + 2\tau) \le T/2
    $
    which implies
    $    
        k \ge \frac{T}{4(\tau + 1)}
        .
    $
\end{proof}

Given the lemma above, we prove \cref{lem:no-noise} by showing that if the algorithm fails, it makes many updates in all of which we have $\norm{\nabla f(x_t)} > \epsilon$. By \cref{lem:loacl_improvement}, this means that in the $T$ time steps of the algorithm, it must decrease the value of $f$ significantly. Since we start at a point in which $f(x_1) \le F$, we must conclude that $T$ cannot be too large.

\begin{proof}[Proof of \cref{lem:no-noise}]
    Combining \cref{eq:descent} with $\eta = 1/(4\beta)$ and \cref{lem:update-lower-bound}, we get that if the algorithm fails with probability $\ge \tfrac12$ then
    \begin{align*}
        F 
        &\ge 
        \sum_{t=1}^T\E \Delta_t
        \ge 
        \frac{1}{16 \beta} \sum_{t=1}^T \E \brk[s]{\ind_t \|\nabla f(x_t)\|^2}
        \ge 
        \frac{1}{16 \beta}\E \brk[s]4{\sum_{t=1}^T  \ind_t \|\nabla f(x_t)\|^2}
        \\
        &\ge 
        \frac{1}{32 \beta}\E \brk[s]4{\sum_{t=1}^T  \ind_t \|\nabla f(x_t)\|^2 \;\Bigg|\; \text{algorithm fails}}
        \ge 
        \frac{\epsilon^2}{32 \beta} \E \brk[s]4{\sum_{t=1}^T  \ind_t \;\Bigg|\; \text{algorithm fails}}
        \ge 
        \frac{\epsilon^2}{32 \beta}\frac{T}{4(\tau + 1)}.
    \end{align*}
    This yields the lemma's statement.
\end{proof}

\subsection{Case (ii): \texorpdfstring{$\boldsymbol{\sigma > \epsilon}$}{}}

This is the ``high noise'' regime. 
For this case, we prove the following guarantee for the convergence of our algorithm.

\begin{lemma} \label{lem:improvements-tradeoff}
    Assume that $\sigma > \epsilon$ and the algorithm fails with probability $\ge \tfrac12$.
    Then, 
    $$
        \sum_{t=1}^{T}\E \Delta_t 
        \ge 
        \frac{T}{500 \beta} \min\brk[c]3{\frac{\epsilon^2}{\tau}, \frac{\epsilon^4}{\sigma^2}}
        .
    $$
    In particular, 
    $$
        T 
        \le 
        500 \beta F \brk3{
        \frac{\tau}{\epsilon^2}
        + 
        \frac{\sigma^2}{\epsilon^4}}
        .
    $$
\end{lemma}

This result is attained using the following observation. Consider the iterate of algorithm at time $t$, $x_t$, and the point at which the gradient was computed $x'_t = x_{t - d_t}$. 
We claim that if the algorithm has not decreased the function value sufficiently during the interval $[t-d_t,t-1]$, then it is likely to trigger a large decline in the function value at time $t$.
Formally, either $\E \Delta_t$ is large, or $\sum_{i=t-d_t}^{t-1} \E \Delta_i$ is large.
To show the claim, we first upper bound the distance $\norm{x_t-x'_t}$ in terms of $\sum_{i=t-d_t}^{t-1} \E \Delta_i$, as shown by the following technical lemma.
    
\begin{lemma} \label{lem:staleness-bound}
For all $t$ and $k$, it holds that
    \[
        \E \norm{x_{t} - x_{t+k}}
        \le 
        \sqrt{\frac{1}{\beta} \sum_{i=t}^{t+k-1} \E \Delta_i} 
        + 
        \frac{4}{\epsilon} \sum_{i=t}^{t+ k - 1} \E \Delta_i
        .
    \]
\end{lemma}

\begin{proof}
    We have 
    \begin{align*}
        \E \|x_{t} - x_{t+k}\| 
        = 
        \eta\E \norm4{\sum_{i=t}^{t+ k- 1 } \ind_i \brk{\nabla f(x'_i) + N_i}} 
        \le 
        \eta\E \norm4{\sum_{i=t}^{t+ k - 1} \ind_i\nabla f(x'_i)}
        +
        \eta \E \norm4{\sum_{i=t}^{t+ k- 1} \ind_i N_i}
        .
    \end{align*}
    We continue bounding the second term above as follows:
    \begin{align*}
        \E \left\| \sum_{i=t}^{t+ k- 1} \ind_i N_i\right\| 
        &\le
        \sqrt{\E \left\| \sum_{i=t}^{t+ k- 1 } \ind_i N_i\right\|^2} 
        \\
        &= 
        \sqrt{\E  \sum_{i=t}^{t+ k- 1 }\sum_{j=t}^{t+ k- 1 } \ind_i \ind_j N_i \cdot N_j} 
        \\
        &= 
        \sqrt{\E  \sum_{i=t}^{t+ k- 1} \ind_i \|N_i\|^2} 
        \tag{$\E[N_i \mid \ind_i, \ind_j,N_j] = 0$ for $i>j$}
        \\
        &\le 
        \sigma \sqrt{\E  \sum_{i=t}^{t+ k- 1}  \ind_i}
        \\
        &\le 
        \frac{\sigma}{\epsilon} \sqrt{\E  \sum_{i=t}^{t+ k- 1}  \ind_i \norm{\nabla f(x'_i)}^2}
        \tag{$\norm{\nabla f(x'_i)} \ge \epsilon$ when $\ind_i = 1$}
        \\
        &\le 
        \frac{\sigma}{\epsilon} \sqrt{ \frac{16 \sigma^2 \beta}{\epsilon^2} \, \sum_{i=t}^{t+ k- 1} \E \Delta_i}
        \tag{\cref{eq:descent}, $\eta = \ifrac{\epsilon^2}{4 \beta \sigma^2}$}
        \\
        &=
        \frac{4 \sigma^2}{\epsilon^2} \sqrt{\beta \sum_{i=t}^{t+ k- 1} \E \Delta_i}
        \\
        &=
        \frac{1}{\eta} \sqrt{\frac{1}{\beta} \sum_{i=t}^{t+ k- 1} \E \Delta_i}
        \tag{$\eta = \ifrac{\epsilon^2}{4 \beta \sigma^2}$}
        ,
    \end{align*}
    and
    \begin{align*}
        \E \norm4{\sum_{i=t}^{t+ k- 1} \ind_i\nabla f(x'_i)}
        &\le 
        \sum_{i=t}^{t+ k- 1} \E \ind_i \norm{\nabla f(x'_i)}
        \\
        &\le 
        \frac{1}{\epsilon} \sum_{i=t}^{t+ k- 1} \E \ind_i \norm{\nabla f(x'_i)}^2 
        \tag{$\norm{\nabla f(x'_i)} \ge \epsilon$ when $\ind_i = 1$}
        \\
        &\le
        \frac{4}{\epsilon \eta} \sum_{i=t}^{t+ k- 1} \E \Delta_i
        \tag{\cref{eq:descent}}
        .
    \end{align*}
    This completes the proof.
\end{proof}

Given the lemma above, it is now clear that if $\sum_{i=t-d_t}^{t-1} \E \Delta_i$ is sufficiently small, then $\E \norm{x_t - x'_t} \ll \ifrac{\epsilon}{\beta}$ which means that the algorithm is likely (with constant probability) to take a step at time $t$.
This argument yields the following.

\begin{corollary} \label{corr:descent-guarantee}
    Assume that the algorithm fails with probability $\ge \tfrac12$.
    If
    $
        \sum_{i=t-d_t}^{t-1}\E \Delta_i 
        < 
        \ifrac{\epsilon^2}{125 \beta}
    $ 
    then 
    $ 
        \E \Delta_t 
        \ge 
        \ifrac{\epsilon^4}{64\sigma^2 \beta}
        .
    $
    In particular,
    $$
        \E \Delta_t  + \frac{1}{2 \tau}\sum_{i=t-d_t}^{t-1}\E \Delta_i 
        \ge 
        \frac{1}{250 \beta} \min\brk[c]3{\frac{\epsilon^2}{\tau}, \frac{\epsilon^4}{\sigma^2}}
        .
    $$
\end{corollary}

\begin{proof}
    If 
    $
        \sum_{i=t-d_i}^{t-1} \E \Delta_i 
        < 
        \ifrac{\epsilon^2}{125 \beta}
        ,
    $
    then 
    $
        \E\|x_{t-d_t}-x_t\| 
        \le 
        \ifrac{\epsilon}{8 \beta}
    $
    by \cref{lem:staleness-bound}.
    By a Markov inequality, with probability $\ge \tfrac{3}{4}$, we have
    $
        \|x_{t-d_t}-x_t\| 
        \le  
        \ifrac{\epsilon}{2 \beta}
        .
    $ 
    Since the probability that $\|\nabla f(x_{t-d_t})\| > \epsilon$ is at least $\frac12$, we get that $\E \ind_t \ge \frac14$. 
    By \cref{lem:loacl_improvement} this implies that 
    $$
        \E\Delta_t 
        \ge
        \frac14 \cdot \frac{\epsilon^2 \cdot \epsilon^2}{16 \sigma^2 \beta} 
        =
        \frac{\epsilon^4}{64\sigma^2 \beta}
        ,
    $$
    which yields our claim.
\end{proof}

We now prove our main claim. We show that if the algorithm fails, then in \emph{all} time steps in which $d_t \le 2 \tau$ (of which there are at least $T/2$), either the algorithm makes a substantial step, or it has made significant updates in the interval $[t-d_t, t-1]$. In any case, the function value must necessarily decrease overall in the $T$ time steps of the algorithm, concluding that $T$ cannot be too large. 

\begin{proof}[Proof of \cref{lem:improvements-tradeoff}]
    We have,
    \[
        \sum_{t=1}^{T}\E \Delta_t 
        \ge 
        \sum_{t : d_t\le 2\tau} \frac{1}{2 \tau}\sum_{i=t-d_t}^{t-1}\E\Delta_i
        .
    \]
    Hence, using \cref{corr:descent-guarantee},
    \begin{align*}
        \sum_{t=1}^{T}\E \Delta_t 
        &\ge
        \frac12 \sum_{t : d_t\le 2\tau} \brk4{\E\Delta_t + \frac{1}{2\tau}\sum_{i=t-d_t}^{t-1}\E\Delta_i} \\
        &\ge 
        \abs!{\set{ t : d_t\le 2\tau }} \, \frac{1}{250 \beta} \min\brk[c]3{\frac{\epsilon^2}{\tau}, \frac{\epsilon^4}{\sigma^2}} \\
        &\ge 
        \frac{T}{2} \, \frac{1}{250 \beta} \min\brk[c]3{\frac{\epsilon^2}{\tau}, \frac{\epsilon^4}{\sigma^2}} \\
        &=
        \frac{T}{500 \beta} \min\brk[c]3{\frac{\epsilon^2}{\tau}, \frac{\epsilon^4}{\sigma^2}}
        ,
    \end{align*}
    where we used Markov's inequality to show that $| \{ t : d_t\le 2\tau\} | \ge \frac12 T$.
\end{proof}

\subsection{Concluding the proof}

\begin{proof}[Proof of \cref{thm:main}]
In the case $\sigma\leq\epsilon$, \cref{lem:no-noise} implies that if $T > \ifrac{128 \beta F (\tau +1)}{\epsilon^2}$ then the algorithms succeeds with probability greater than $1/2$, which yields the theorem in this case.
Similarly, \cref{lem:improvements-tradeoff} gives our claim in the case when $\sigma>\epsilon$.
%
\end{proof}

\section{Experiments}
To illustrate the robustness and efficacy of \algo, we present
a comparison between the performance of SGD versus \algo under various delay distributions.
In particular, we show that \algo requires significantly less iterations to reach a fixes goal and is more robust to varying delay distributions.


\subsection{Setup}

The main goal of our experimental setup is to be reproducible. 
For that end, the experimentation is done in two phases. First, we perform a simulation to determine the delay $d_t$ at each iteration without actually computing any gradients:\footnote{Note that up to the training data ordering a computation of $T$ steps of \algo or SGD is uniquely determined by the starting state $x_1$ and the sequence $\{t - d_t\}_{t=1 \ldots T}$.}
this is done by simulating $N$ concurrent worker threads sharing and collectively advancing a global iteration number, where each worker repeatedly records the current global iteration number $t_{\text{start}}$, waits a random amount of time from a prescribed Poisson distribution, then records the new global iteration number $t = t_{\text{end}}$ and the difference $d_t = t_{\text{end}} - t_{\text{start}}$, and increases the global iteration number. 
This information (a {\em delay schedule}) is calculated once for each tested scheme (differing in the number of workers and random distribution, as detailed below), and is stored for use in the second phase. 

In the second phase of the experiments, the algorithms SGD and \algo are executed for each delay schedule. Here, at every iteration the gradient is computed (if needed) and is kept until its usage as dictated by the schedule (and then applied at the appropriate global iteration number).  
As a result of this configuration, we get a fully reproducible set of experiments, where the algorithms performance may be compared as they are executed over identical delay series of identical statistical properties.

We created four different delay schedules:
A baseline schedule (A) using $N=10$ workers and sampling the simulated wait from a Poisson distribution (this schedule serves to compare \algo and SGD in a setting of relatively small delay variance) and schedules (B) (C) and (D) all using $N=75$ workers and sampling the simulated wait from bi-modal mixtures of Poisson distributions of similar mean but increasing variance respectively.\footnote{See the \supplementaryMaterial for specific parameter values and implementation details.}
See Figure \ref{fig:delay} in the \supplementaryMaterial for an illustration of the delay distributions of the four delay schedules used.

All training is performed on the standard CIFAR-10 dataset~\citep{krizhevsky2009learning} using a ResNet56 with $9$ blocks model~\citep{he2016deep} and implemented in TensorFlow~\citep{tensorflow2015-whitepaper}. 
We compare \algo (\cref{alg:sgd-with-delays}) to the SGD algorithm
which unconditionally updates the state $x_t$ given the stochastic delayed gradient $g_t$ (recall that $g_t$ is the stochastic gradient at state $x_{t-d_t}$).

For both algorithms, instead of a constant learning rate $\eta$ we use a piecewise-linear learning rate schedule as follows: we consider
a baseline $\eta_0$ piecewise-linear learning rate schedule\footnote{With rate changes at three achieved accuracy points 0.93, 0.98, and 0.99.} that achieves optimal performance in a synchronous distributed optimization setting (that is, for $d_t \equiv 0)$\footnote{This is also the best performance achievable in an asynchronous setting.}
and search (for each of the four delay schedules and each algorithm -- to compensate for the effects of delays) for the best multiple of the baseline rate and the best first rate-change point.
Alternatively, we also used a cosine decay learning rate schedule (with the duration of the decay as meta parameters).
Another meta-parameter we optimize is the threshold $\epsilon/(2\beta)$ in \cref{ln:test-update} of \algo. Batch size 64 was used throughout the experiments.
Note that although use chose the threshold value $\ifrac{\epsilon}{2 \beta}$ by an exhaustive search, in practice, a good choice can be found by logging the distance values during a typical execution and choosing a high percentile value. See \cref{S:efficient} for more details.

\begin{figure*}
\centering
\subfigure{\includegraphics[width=.49\textwidth]{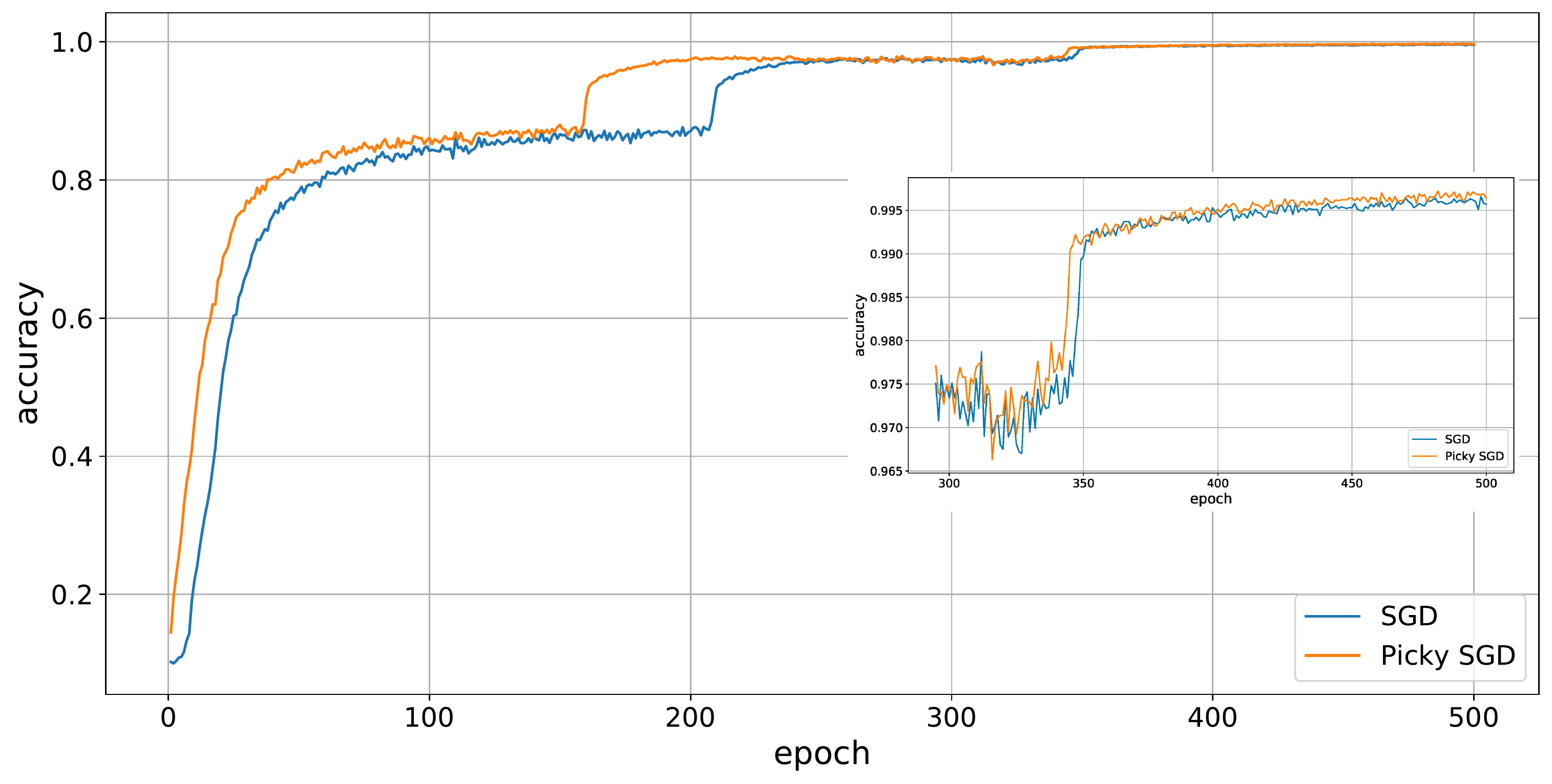}}
\subfigure{\includegraphics[width=.49\textwidth]{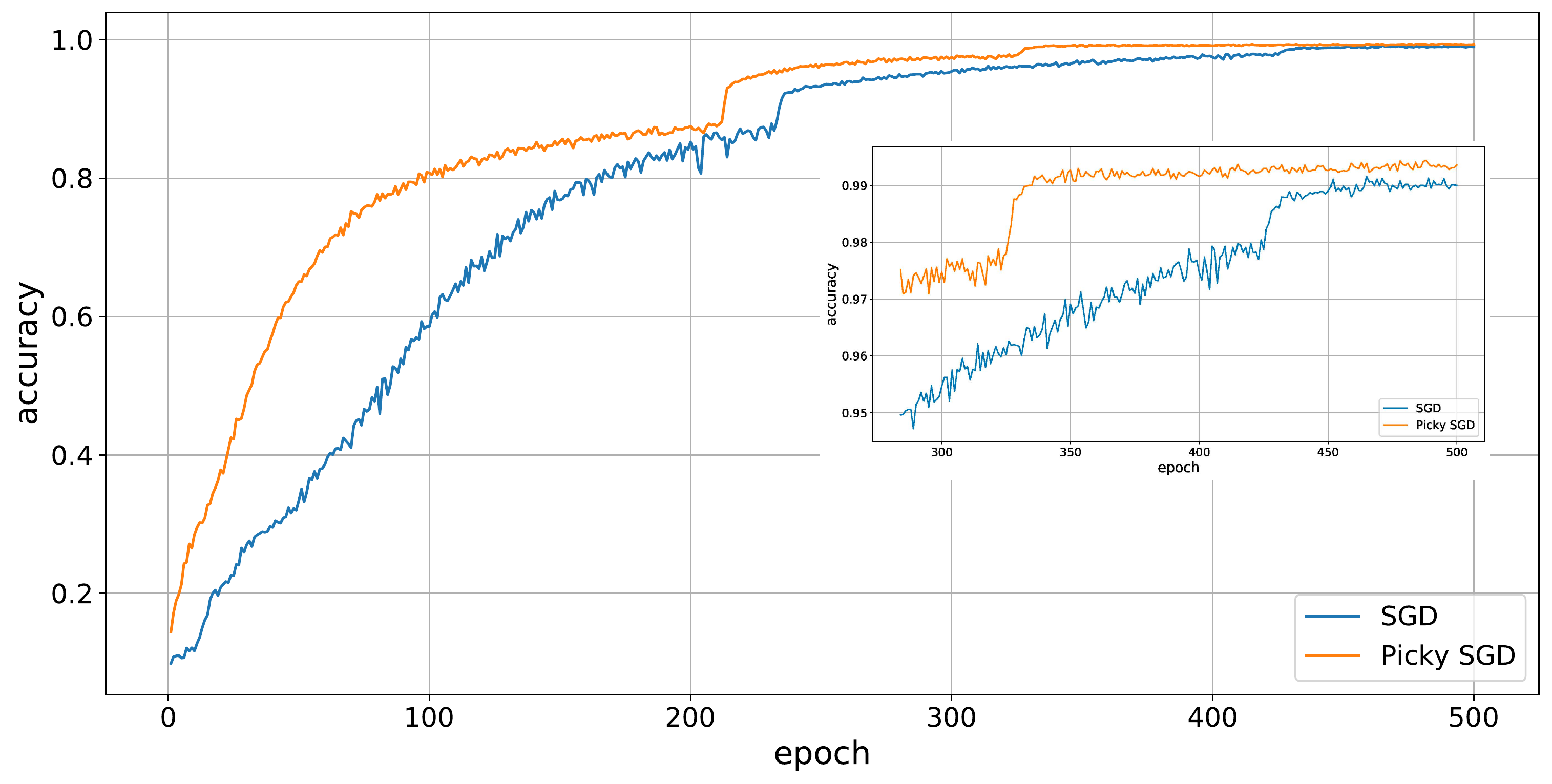}}
\subfigure{\includegraphics[width=.49\textwidth]{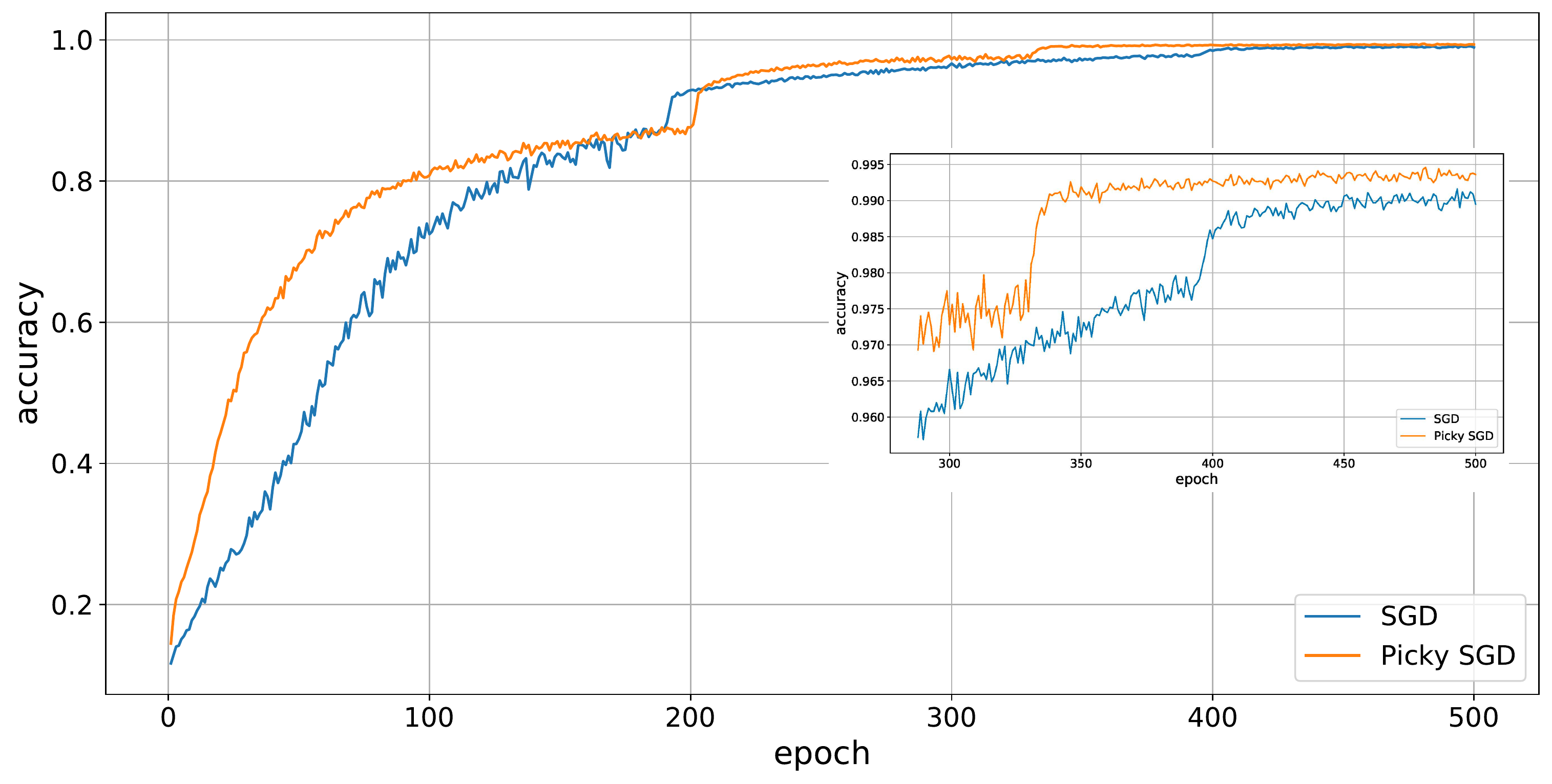}}
\subfigure{\includegraphics[width=.49\textwidth]{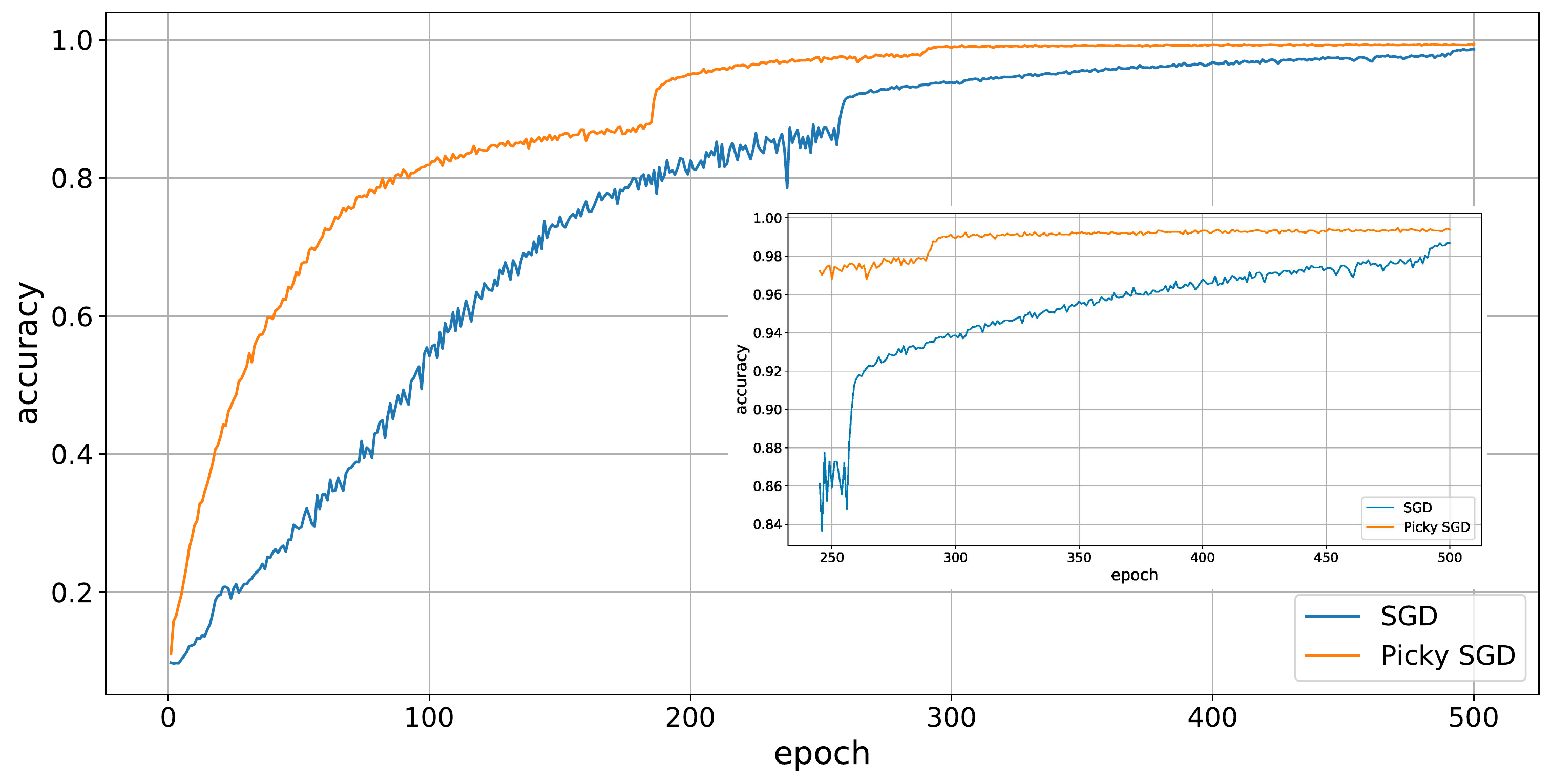}}
\vspace{-0.15in}
\caption{Accuracy trajectory (with a zoom-in on the tail of the convergence) over train epochs for the four delay schedules of \cref{fig:delay}, respectively: the key metrics (reported in \cref{table:results}) for each trajectory are epochs to reach 0.99 accuracy (the number of epochs required to reach the 0.99 accuracy mark) and the baseline learning rate multiplier $\ifrac{\eta}{\eta_0}$.}
\label{fig:algs}
\end{figure*}

\subsection{Results}
The accuracy trajectory for the best performing combination of parameters of each algorithm for each of the four delay schedules
is shown in \cref{fig:algs} and summarized in \cref{table:results}. 
Clearly, \algo significantly outperforms SGD in terms of the final accuracy and the number of epochs it takes to achieve it. 
We also emphasize that the generalization performance (that is, the evaluation accuracy as related to the training accuracy) was not observed to vary across delay schedules or the applied algorithms (see e.g., \cref{fig:algseval2} in the \supplementaryMaterial), and that the nature of the results is even more pronounced when using the alternative cosine decay learning rate schedule (see \cref{fig:algseval_cd} in the \supplementaryMaterial).   
Specific details of the meta parameters used, and additional performance figures are reported in \cref{sec:exp_detail}.

\begin{table}[h!]
\centering
\caption{Summary of the key metrics from \cref{fig:algs}, for each of the four delay schedules A, B, C, and D 
.}
\vspace{0.1in}
\begin{tabular}{ |p{0.3cm}||p{1.0cm}|p{1.0cm}||p{1.0cm}|p{1.0cm}|}
 \hline
    & \multicolumn{2}{|c||}{Epochs to 0.99\%} 
    & \multicolumn{2}{|c|}{LR multiplier $(\ifrac{\eta}{\eta_0})$} \\
 \hline
  & Picky SGD & SGD & Picky SGD & SGD \\
  \hline
  \hline
 A& 344 & 350 & 0.5 & 0.5\\
 \hline
 B& 333 & 451 & 0.2 & 0.05\\
 \hline
 C& 337 & 438 & 0.2 & 0.05\\
 \hline
 D& 288 & 466 & 0.2 & 0.05\\
 \hline
\end{tabular}
\label{table:results}
\end{table}
\subsection{Discussion}

We first observe that while the number of epochs it takes \algo to reach the target accuracy mark is almost the same across the delay schedules (ranging from $288$ to $344$), SGD requires significantly more epochs to attain the target accuracy (ranging from $350$ up to $466$ for the highest variance delay schedule)---this is consistent with the average-delay bound dependence of \algo (as stated in \cref{thm:main}) compared to the max-delay bound dependence of SGD. Furthermore, the best baseline learning rate multiplier meta-parameter for \algo is the same~(0.2) across all high-variance delay schedules, while the respective meta parameter for SGD is significantly smaller~(0.05) and sometimes varying, explaining the need for more steps to reach the target and evidence of \algo superior robustness.



\subsection*{Acknowledgements}

AD is partially supported by the Israeli Science Foundation (ISF) grant no.~2258/19.
TK is partially supported by the Israeli Science Foundation (ISF) grant no.~2549/19, by the Len Blavatnik and the Blavatnik Family foundation, and by the Yandex Initiative in Machine Learning.

\bibliography{references}

\appendix
\onecolumn

\supplementaryTitle



\section{\algo in the Convex Case} \label{sec:convex}

The algorithm for the convex case is a variant of \cref{alg:sgd-with-delays} and is displayed in \cref{alg:sgd-convex-with-delays}. In fact, the only difference from \cref{alg:sgd-with-delays} is in \cref{ln:convex-test} where the we use a threshold of $\sqrt{\ifrac{\epsilon}{8 \beta}}$ instead of $\ifrac{\epsilon}{2 \beta}$.

\begin{algorithm}[ht]
    \caption{\algo for Convex Objectives \label{alg:sgd-convex-with-delays}}
    \begin{algorithmic}[1]
        \STATE {\bf input}: 
            learning rate $\eta$, 
            target accuracy $\epsilon$.
        \FOR{$t=1,\ldots,T$}
            \STATE {\bf receive} delayed stochastic gradient $g_t$ and point $x_{t-d_t}$ such that $\E_t[g_t] = \nabla f(x_{t-d_t})$.
            \IF{$\norm{x_t - x_{t-d_t}} \le \sqrt{\ifrac{\epsilon}{8 \beta}}$} \alglinelabel{ln:convex-test}
                \STATE {\bf update:} $x_{t+1} = x_t - \eta g_t$.
            \ELSE 
                \STATE {\bf pass:} $x_{t+1} = x_t$.
            \ENDIF
            \STATE {\bf query} stochastic gradient at $x_{t+1}$
        \ENDFOR
    \end{algorithmic}
\end{algorithm}

Our guarantee for the algorithm is as follows.

\begin{theorem} \label{thm:convex-main}
    Let $x_\star = \argmin_{x \in \reals^d} f(x)$.
    Suppose that \cref{alg:sgd-convex-with-delays} is initialized at $x_1 \in \reals^d$ with $\norm{x_1 - x_\star} \leq F$ and run with
    \[
        T 
        \geq 
        1600 F^2 \brk3{\frac{\sigma^2}{\epsilon^2} + \frac{\beta(\tau+1)}{\epsilon}}
        ,\quad
        \eta 
        = 
        \min \brk[c]3{\frac{1}{16 \beta}, \frac{\epsilon}{8 \sigma^2}}
        ,
    \]
    where $\tau$ be the average delay, i.e., $\tau = (\ifrac{1}{T}) \sum_{t=1}^T d_t$.
    Then, with probability at least $\frac12$, there is some $1 \leq t \leq T$ for which $f(x_t) - f(x_\star) \le \epsilon$.
\end{theorem}

Note that we ensure the success of the algorithm only with probability $\frac12$, but our guarantee can be easily converted to high probability in the same manner as done for \cref{alg:sgd-with-delays}.

\subsection{Analysis}

The analysis proceeds similarly to that of \cref{thm:main}.
We analyze a variant of the algorithm that makes an update if $\norm{x_t-x'_t} > \sqrt{\ifrac{\epsilon}{8 \beta}}$ or $f(x_t) - f(x_\star) \le \epsilon$ then $x_{t+1} = x_t$. Else, $x_{t+1} = x_t - \eta g_t$.
As in the proof of \cref{thm:main}, this variant is impossible to implement, but the guarantee of \cref{thm:convex-main} is valid for this variant if and only if it is valid for the original algorithm.

We next introduce a bit of additional notation. 
We denote by $\ind_t$ the indicator of event that the algorithm performed an update at time $t$. 
Namely, 
$$
    I_t 
    = 
    \ind\set!{ \| x_t-x'_t \| \le \sqrt{\ifrac{\epsilon}{8\beta}} ~\;\text{and}\; f(x_t)- f(x_\star) > \epsilon}
    .
$$
Note that $\ind_t=1$ implies that $f(x_s)- f(x_\star) \ge \epsilon$ for all $s =
1,\ldots,t$. Further, we denote by $\Delta_t = \norm{x_t - x_\star}^2 - \norm{x_{t+1} - x_\star}^2$ the
improvement at time $t$. 
Since we assume $\norm{x_1 - x_\star} \le F$, we have
that for all $t$, 
$$ 
    \sum_{i=1}^t \Delta_i 
    = 
    \norm{x_1 - x_\star}^2 - \norm{x_{t+1} - x_\star}^2
    \leq 
    F^2
    .
$$

Moreover, towards the proof we need the following that holds for any $\beta$-smooth convex function  \citep[see, e.g.,][]{nesterov2018lectures}:
\begin{equation} \label{eq:convex-fact}
    \frac{1}{2\beta} \norm{\nabla f(x)}^2 \le f(x) - f(x_\star) \le \nabla f(x) \dotp \brk{f(x)-f(x_\star)}
    ,
    \quad \text{for any } x \in \reals^d.
\end{equation}


The following lemma is an analog of \cref{lem:loacl_improvement} which shows that whenever the algorithm makes an update, the squared distance to the optimum $\norm{x_t-x_\star}^2$ decreases significantly.

\begin{lemma} \label{lem:convex-loacl_improvement}
    Let $N\in \reals^d$ be a random zero-mean vector with $\E\norm{N}^2 \le \sigma^2$. Fix $x,x'\in \reals^d$ with $\norm{x-x'} \le \ifrac{\epsilon}{4 \beta}$ and $f(x') - f(x_\star) > \epsilon$. 
    Then,
    \begin{align*}
        \E\brk[s]{\norm{x - \eta \brk{\nabla f(x') + N} - x_\star}^2 - \norm{x-x_\star}^2}
        \le 
        -\frac{\eta}{2} \norm{\nabla f(x')}^2
        + 
        \frac{\eta^2}{2} \brk{\sigma^2 + \norm{\nabla f(x')}^2}
        .
    \end{align*}
    In particular, for our choice of $\eta$, we have
    \begin{align} \label{eq:convex-descent}
        \E\brk[s]{\norm{x - \eta \brk{\nabla f(x') + N} - x_\star}^2 - \norm{x-x_\star}^2}
        \le
        -\frac{\eta}{4} \brk{f(x') - f(x_\star)}
        .
    \end{align}
\end{lemma}

\begin{proof}
    We have,
    \begin{align*}
        \norm{x - \eta \brk{\nabla f(x') + N} - x_\star}^2 - \norm{x-x_\star}^2
        =
        -\eta \brk{\nabla f(x') + N} \dotp (x-x_\star)
        + 
        \eta^2 \norm{\nabla f(x') + N}^2
        .
    \end{align*}
    Taking expectation over $N$ we get
    \begin{align*}
        \E\brk[s]{\norm{x - \eta \brk{\nabla f(x') + N} - x_\star}^2 - \norm{x-x_\star}^2}
        &\leq
        -\eta \nabla f(x') \dotp (x-x_\star)
        + 
        \eta^2 \brk{\norm{\nabla f(x')}^2 + \sigma^2}
        \\
        &=
        -\eta \nabla f(x') \dotp (x'-x_\star)
        +
        \eta \nabla f(x') \dotp (x'-x)
        + 
        \eta^2 \norm{\nabla f(x')}^2 
        + 
        \eta^2 \sigma^2
        \\
        &\le
        -\eta \nabla f(x') \dotp (x'-x_\star)
        +
        \eta \norm{\nabla f(x')} \norm{x'-x}
        + 
        \eta^2 \norm{\nabla f(x')}^2 
        + 
        \eta^2 \sigma^2
        \\
        &\leq
        -\eta \brk{f(x') - f(x_\star)}
        +
        \eta \sqrt{2 \beta \brk{f(x') - f(x_\star)}} \, \norm{x'-x} \\
        &\qquad + 
        2 \beta \eta^2 \brk{f(x') - f(x_\star)}
        + 
        \eta^2 \sigma^2
        ,
    \end{align*}
    by using \cref{eq:convex-fact}.
    
    Now, since $\epsilon \le f(x') - f(x_\star)$,
    $$
        \|x-x'\| \leq \sqrt{\frac{\epsilon}{8 \beta}} 
        \leq 
        \sqrt{\frac{f(x') - f(x_\star)}{8\beta}},
    $$ 
    and we have
    \begin{align*}
        \E\brk[s]{\norm{x - \eta \brk{\nabla f(x') + N} - x_\star}^2 - \norm{x-x_\star}^2}
        \le 
        -\frac12 \eta \brk{f(x') - f(x_\star)}
        + 
        2 \beta \eta^2 \brk{f(x') - f(x_\star)}
        + 
        \eta^2 \sigma^2
        .
    \end{align*}
    
    If $2 \beta \epsilon \ge \sigma^2$ then $\sigma^2 \le 2 \beta (f(x')-f(x_\star))$. This, with $\eta = \ifrac{1}{16\beta}$, yields \cref{eq:convex-descent}.
    If $2 \beta \epsilon < \sigma^2$ and $\eta = \ifrac{\epsilon}{8 \sigma^2}$. Plugging that in instead and using $f(x')-f(x_\star) \ge \epsilon$, gets us \cref{eq:convex-descent}.
\end{proof}

Note that by \cref{lem:loacl_improvement} we have that $\E\Delta_t \ge 0$. 

\subsubsection{Case (i): \texorpdfstring{${\sigma^2\le 2 \beta \epsilon}$}{}}

We now handle the ``low noise'' and ``high noise'' regimes differently, beginning with the ``low noise'' regime. 
Recall that by \cref{lem:update-lower-bound} the algorithm makes at least $\Omega (\ifrac{T}{\tau})$ updates.
This yields the following.

\begin{lemma} \label{lem:convex-no-noise}
Suppose that $\sigma^2 \le 2 \beta \epsilon$ and the algorithm fails with probability $\geq \tfrac12$. 
Then $T \le \ifrac{512 \beta F^2 (\tau +1)}{\epsilon}$.
\end{lemma}

\begin{proof}
    Combining \cref{lem:update-lower-bound,lem:convex-loacl_improvement} we get that if the algorithm fails with probability $\ge \tfrac12$ then
    \begin{align*}
        F^2
        &\ge 
        \sum_{t=1}^T\E \Delta_t
        \\
        &\ge 
        \frac{1}{64 \beta} \E \brk[s]4{\sum_{t=1}^T \ind_t \brk{f(x_t) - f(x_\star)}} \\
        &\ge 
        \frac{1}{128 \beta}\E \brk[s]4{\sum_{t=1}^T \ind_t \brk{f(x_t) - f(x_\star)} \;\Bigg|\; \text{algorithm fails}}
        \\
        &\ge 
        \frac{\epsilon}{128 \beta} \E \brk[s]4{\sum_{t=1}^T  \ind_t \;\Bigg|\; \text{algorithm fails}}
        \\
        &\ge 
        \frac{\epsilon}{128 \beta}\frac{T}{4(\tau + 1)}. \qedhere
    \end{align*}
\end{proof}

\subsubsection{Case (ii): \texorpdfstring{${\sigma^2 > 2 \beta \epsilon}$}{}}

This is the ``high noise'' regime.
For this case, we prove the following guarantee for the convergence of our algorithm.

\begin{lemma} \label{lem:convex-improvements-tradeoff}
    Assume that $\sigma^2 > 2 \beta \epsilon$ and the algorithm fails with probability $\ge \tfrac12$.
    Then, 
    $$
        \sum_{t=1}^{T}\E \Delta_t 
        \ge 
        \frac{T}{1600} \min\brk[c]3{\frac{\epsilon}{\tau \beta}, \frac{\epsilon^2}{\sigma^2}}
        .
    $$
    In particular, 
    $$
        T 
        \le 
        1600 F^2 \brk3{
        \frac{\tau \beta}{\epsilon}
        + 
        \frac{\sigma^2}{\epsilon^2}}
        .
    $$
\end{lemma}

Similarly to \cref{lem:improvements-tradeoff}, this result is attained by showing that either $\E \Delta_t$ is large, or $\sum_{i=t-d_t}^{t-1} \E \Delta_i$ is large.
To prove the claim, we first upper bound the distance $\norm{x_t-x'_t}$ in terms of $\sum_{i=t-d_t}^{t-1} \E \Delta_i$, as shown by the following lemma.
    
\begin{lemma} \label{lem:convex-staleness-bound}
For all $t$ and $k$, it holds that
    \[
        \E \norm{x_{t} - x_{t+k}}
        \le 
        \sqrt{\sum_{i=t}^{t+k-1} \E \Delta_i} 
        + 
        \sqrt{\frac{32 \beta}{\epsilon}} \sum_{i=t}^{t+ k - 1} \E \Delta_i
        .
    \]
\end{lemma}

\begin{proof}
    We have 
    \begin{align*}
        \E \|x_{t} - x_{t+k}\| 
        = 
        \eta\E \norm4{\sum_{i=t}^{t+ k- 1 } \ind_i \brk{\nabla f(x'_i) + N_i}}
        \le 
        \eta\E \norm4{\sum_{i=t}^{t+ k - 1} \ind_i\nabla f(x'_i)}
        +
        \eta \E \norm4{\sum_{i=t}^{t+ k- 1} \ind_i N_i}
        .
    \end{align*}
    We continue bounding the second term above as follows:
    \begin{align*}
        \E \left\| \sum_{i=t}^{t+ k- 1} \ind_i N_i\right\| 
        &\le
        \sqrt{\E \left\| \sum_{i=t}^{t+ k- 1 } \ind_i N_i\right\|^2} 
        \\
        &= 
        \sqrt{\E  \sum_{i=t}^{t+ k- 1 }\sum_{j=t}^{t+ k- 1 } \ind_i \ind_j N_i \cdot N_j} 
        \\
        &= 
        \sqrt{\E  \sum_{i=t}^{t+ k- 1} \ind_i \|N_i\|^2} 
        \tag{$\E[N_i \mid \ind_i, \ind_j,N_j] = 0$ for $i>j$}
        \\
        &\le 
        \sqrt{\sigma^2 \E  \sum_{i=t}^{t+ k- 1}  \ind_i}
        \\
        &\le 
        \sqrt{\frac{\sigma^2}{\epsilon} \E  \sum_{i=t}^{t+ k- 1}  \ind_i \brk{f(x'_i) - f(x_\star)}}
        \tag{$f(x'_i) - f(x_\star) \ge \epsilon$ when $\ind_i = 1$}
        \\
        &\le 
        \sqrt{\frac{\sigma^2}{\epsilon} \cdot \frac{32 \sigma^2}{\epsilon} \, \sum_{i=t}^{t+ k- 1} \E \Delta_i}
        \tag{\cref{eq:convex-descent}}
        \\
        &\le
        \frac{8 \sigma^2}{\epsilon}  \sqrt{\sum_{i=t}^{t+ k- 1} \E \Delta_i}
        \\
        &=
        \frac{1}{\eta} \sqrt{\sum_{i=t}^{t+ k- 1} \E \Delta_i}
        \tag{$\eta = \ifrac{\epsilon}{8 \sigma^2}$}
        ,
    \end{align*}
    and
    \begin{align*}
        \E \norm4{\sum_{i=t}^{t+ k- 1} \ind_i\nabla f(x'_i)}
        &\le 
        \sum_{i=t}^{t+ k- 1} \E \ind_i \norm{\nabla f(x'_i)}
        \\
        &\le 
        \sum_{i=t}^{t+ k- 1} \E \ind_i \sqrt{2 \beta \brk{f(x'_i) - f(x_\star)}}
        \tag{$f$ convex and $\beta$-smooth}
        \\
        &\le 
        \sqrt{\frac{2 \beta}{\epsilon}} \sum_{i=t}^{t+ k- 1} \E \ind_i \brk{f(x'_i) - f(x_\star)}
        \tag{$f(x'_i) - f(x_\star) \ge \epsilon$ when $\ind_i = 1$}
        \\
        &\le
        \sqrt{\frac{32 \beta}{\epsilon}} \cdot \frac{1}{\eta} \sum_{i=t}^{t+ k- 1} \E \Delta_i
        \tag{\cref{eq:convex-descent}}
        .
    \end{align*}
    This completes the proof.
\end{proof}

Given the lemma above, it is now clear that if $\sum_{i=t-d_t}^{t-1} \E \Delta_i$ is sufficiently small, then $\E \norm{x_t - x'_t} \ll \sqrt{\frac{\epsilon}{\beta}}$ which means that the algorithm is likely (with constant probability) to take a step at time $t$.
This argument yields the following.

\begin{corollary} \label{corr:convex-descent-guarantee}
    Assume that the algorithm fails with probability $\ge \tfrac12$.
    If
    $
        \sum_{i=t-d_t}^{t-1}\E \Delta_i 
        < 
        \ifrac{\epsilon}{400 \beta}
    $ 
    then 
    $ 
        \E \Delta_t 
        \ge 
        \ifrac{\epsilon^2}{128 \sigma^2}
        .
    $
    In particular,
    $$
        \E \Delta_t  + \frac{1}{2 \tau}\sum_{i=t-d_t}^{t-1}\E \Delta_i 
        \ge 
        \frac{1}{800} \min\brk[c]3{\frac{\epsilon}{\tau \beta}, \frac{\epsilon^2}{\sigma^2}}
        .
    $$
\end{corollary}

\begin{proof}
    If 
    $
        \sum_{i=t-d_i}^{t-1} \E \Delta_i 
        < 
        \ifrac{\epsilon}{400 \beta}
        ,
    $
    then 
    $
        \E\|x_{t-d_t}-x_t\| 
        \le 
        \sqrt{\ifrac{\epsilon}{128 \beta}}
    $
    by \cref{lem:convex-staleness-bound}.
    By a Markov inequality, with probability $\ge \tfrac{3}{4}$, we have
    $
        \|x_{t-d_t}-x_t\| 
        \le  
        \sqrt{\ifrac{\epsilon}{8 \beta}}
        .
    $ 
    Since the probability that $f(x_{t-d_t}) - f(x_\star) > \epsilon$ is at least $\frac12$, we get that $\E \ind_t \ge \frac14$. 
    Finally, by \cref{lem:convex-loacl_improvement} this implies that 
    $$
        \E\Delta_t 
        \ge
        \frac14 \cdot \frac{\epsilon}{32 \sigma^2} \cdot \epsilon
        =
        \frac{\epsilon^2}{128 \sigma^2}
        \ge
        \frac{\epsilon^2}{800 \sigma^2}
        ,
    $$
    which yields our claim.
\end{proof}

We now prove our main claim.

\begin{proof}[Proof of \cref{lem:convex-improvements-tradeoff}]
    We have,
    \[
        \sum_{t=1}^{T}\E \Delta_t 
        \ge 
        \sum_{t : d_t\le 2\tau} \frac{1}{2 \tau}\sum_{i=t-d_t}^{t-1}\E\Delta_i
        .
    \]
    Hence, using \cref{corr:convex-descent-guarantee},
    \begin{align*}
        \sum_{t=1}^{T}\E \Delta_t 
        &\ge
        \frac12 \sum_{t : d_t\le 2\tau} \brk4{\E\Delta_t + \frac{1}{2\tau}\sum_{i=t-d_t}^{t-1}\E\Delta_i}
        \ge 
        \abs!{\set{ t : d_t\le 2\tau }} \, \frac{1}{800} \min\brk[c]3{\frac{\epsilon}{\tau \beta}, \frac{\epsilon^2}{\sigma^2}}
        \ge 
        \frac{T}{2} \cdot \frac{1}{800} \min\brk[c]3{\frac{\epsilon}{\tau \beta}, \frac{\epsilon^2}{\sigma^2}} \\
        &=
        \frac{T}{1600} \min\brk[c]3{\frac{\epsilon}{\tau \beta}, \frac{\epsilon^2}{\sigma^2}}
        ,
    \end{align*}
    where we used Markov's inequality to show that $| \{ t : d_t\le 2\tau\} | \ge \frac12 T$.
\end{proof}

\subsubsection{Concluding the proof}

\begin{proof}[Proof of \cref{thm:convex-main}]
    In the case $\sigma^2\leq 2 \beta \epsilon$, \cref{lem:convex-no-noise} implies that if $T > \ifrac{512 \beta F^2 (\tau +1)}{\epsilon}$, then the algorithms succeeds with probability greater than $1/2$, which yields the theorem in this case.
    Similarly, \cref{lem:convex-improvements-tradeoff} gives our claim in the case when $\sigma^2>2 \beta \epsilon$.
\end{proof}
 
\section{Details of Experiments} \label{sec:exp_detail}

\subsection{Simulation method}
In this section, we describe in detail the simulation environment that we used to compare the performance of distributed optimization algorithms under various heterogeneous distributed computation environments.
Our simulation environment has two components:
One, for generating the order in which the gradients are applied to the parameter state (see \cref{alg:simgenmaster}, \simgenmaster) and another that given that order carries out the distributed optimization computation (sequentially, in a deterministic manner, see \cref{alg:simgalg}, \simalg).

\subsubsection{\simgenmaster}
We consider a sequential distributed computation, where a shared state is updated by concurrent workers, and the update depends on the state history alone - that is, the update at state $S_w$ (at {\em stage} $w$, where 
$w = 1\ldots T$) is a function of the state trajectory $S_0, \ldots S_w$. Specifically, we address the case where the update is a function of a single state $S_r(w), \ \  r(w) \leq w$, from the trajectory. 
Crucially, the sequence 
$\{r(w)\}_{w=1\ldots T}$ uniquely determines the outcome $S_T$ of the computation. This setting captures a distributed gradient descent computation, where the gradient applied at state $S_w$ was computed at state $S_{r(w)}$. Indeed, up to data batching order, different optimization algorithms (SGD and \algo, in our case) may be compared in a deterministic manner over the same generated sequence 
$\{r(w)\}_{w=1\ldots T}$, of certain predefined statistical properties.

We generate a sequence $Q = \{(r(w), w)\}_{w = 1 \ldots T}$ in \cref{alg:simgenmaster}
by simulating concurrent workers that share a state $S$. Each worker (performing \cref{alg:simgen}), at each iteration creates a pair 
$(r,w)$ appended to $Q$ as follows: instead of actually performing the (gradient) computation, the worker merely records the state $r\gets S$, waits for a random period, records the state again $w \gets S$ and advances it (note that during the wait, the shared state may have been advanced by another worker). Note that the nature of the random wait determines the statistical properties of the generated sequence $Q$.

\begin{algorithm}[ht]
    \caption{\simgenworker 
    \label{alg:simgen}}
    \begin{algorithmic}[1]
        \STATE {\bf shared variables}:
            stage $S$,
            sequence $Q$,
            tasks queue $T$.
        \WHILE{$T > 0$}
            \STATE {$T \gets T - 1$}
            \STATE {$r \gets S$}
            \STATE {Wait for a random period}
            \COMMENT{gradient computation}
            \STATE {$w \gets S$, $S \gets S + 1$} \COMMENT{atomically}
            \STATE {Wait for a random period}
            \COMMENT{gradient update}
            \STATE {$Q$.append($[r,w]$)}
        \ENDWHILE
    \end{algorithmic}
\end{algorithm}

\begin{algorithm}[ht]
    \caption{\simgenmaster 
    \label{alg:simgenmaster}}
    \begin{algorithmic}[1]
        \STATE {\bf input}:
            number of workers $n$,
            number of steps $t$,
        \STATE {\bf Initialize shared variables}:
            stage $S \gets 0$,
            sequence $Q \gets []$,
            tasks queue $T \gets t$.
        \STATE {\bf spawn} $n$ workers.
        \STATE {\bf wait until} $T=0$.
        \STATE {\bf return} Sorted $Q$.
    \end{algorithmic}
\end{algorithm}

\subsubsection{\simalg}

Given a sorted\footnote{Lexicographical order. This is the reason for generating pairs in \cref{alg:simgenmaster} and not just $\{r(w)\}_{w = 1 \ldots T}$. Note that every $w$ in $Q$ appears exactly once.}
sequence $Q = \{(r(w), w)\}_{w = 1 \ldots T}$, \cref{alg:simgalg} simulates the distributed optimization computation (of the state $X$ of a given model $M$) by sequentially considering pairs $(r,w)$ from $Q$.

The algorithm maintains a computation stage $S$, a map $G$, where $G[r]$ holds the gradient $g$ computed at stage $r$ (and the state $x$ for which the gradient was computed), and a map $F$, where $F[w]$ is the stage $r$ in which the gradient (to be applied at stage $w$) was computed.

For each pair $(r, w)$ in $Q$, the algorithm updates the map $F$ accordingly, and then, as long as the computation stage $S < r$, uses the information in the maps $G$ and $F$ to apply gradients to the computation state $X$ accordingly and advance the stage $S$, until it reaches $r$. The gradients are applied through the optimization algorithm (i.e., SGD or \algo) which is passed as an input to the simulation.
After the above {\em catch up} phase, the gradient at stage $r$ (Note that at this point $r = S$) is computed (this only happens for the first appearance of $r$) using a batch sampled from the input data set $D$, and
the map $G$ is updated accordingly with the computed gradient $g$ and the computation state $X$.

Note that during the {\em catch up} phase, since $Q$ is sorted, if $r > S$ then $F[S] = r' < S < r$ so the gradient at $r'$ was already computed in a previous iteration (at the first time $r'$ appeared in a pair $(r',S)$ in $Q$). Moreover, $F[S]$ was updated at that same previous iteration, establishing correctness.   

\begin{algorithm}[ht]
    \caption{\simalg 
    \label{alg:simgalg}}
    \begin{algorithmic}[1]
        \STATE {\bf input}:
            Sorted Sequence $Q$.
            Model $M$ with state $X$ to be optimized using algorithm $Alg$ over train data $D$.
        \STATE {\bf Initialize variables}:
            computation stage $S \gets 0$,
            gradients map $G \gets \{\}$,
            apply stage map $F \gets \{\}$.
        \FOR{$[r,w]$ in $Q$}    
            \STATE {$F[w] \gets r$}.
            \WHILE {$S < r$}
                \STATE {$g,x \gets G[(F[S],S)]$}
                \STATE {$X \gets$} $Alg$ apply gradient ($g$, computed at state $x$).
                \STATE {$S \gets S + 1$}
            \ENDWHILE
            \STATE {$d \gets$} data batch from $D$
            \STATE {$g \gets$} model $M$ gradient at $X$ for data batch $d$.
            \STATE {$G[(S, w)] \gets [g,X]$}
        \ENDFOR
        \STATE {\bf return $X$} \COMMENT{optimal state for $M$ over data set $D$ as computed by $Alg$}.
    \end{algorithmic}
\end{algorithm}

\subsection{Experiments settings}
\subsubsection{Generated sequences}

\begin{figure*}[h]
\centering
\subfigure{\includegraphics[width=.49\textwidth]{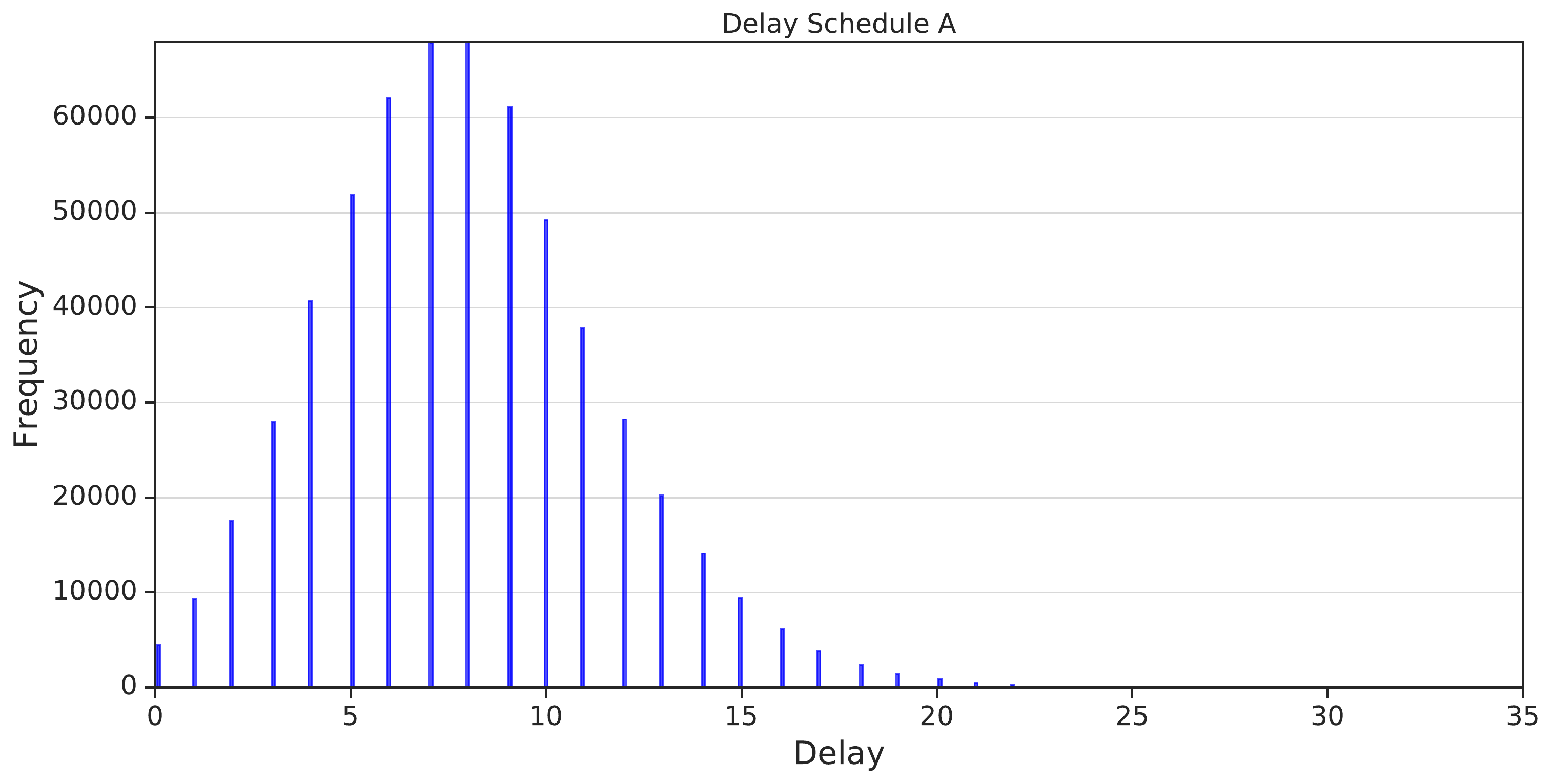}\hfill}
\subfigure{\includegraphics[width=.49\textwidth]{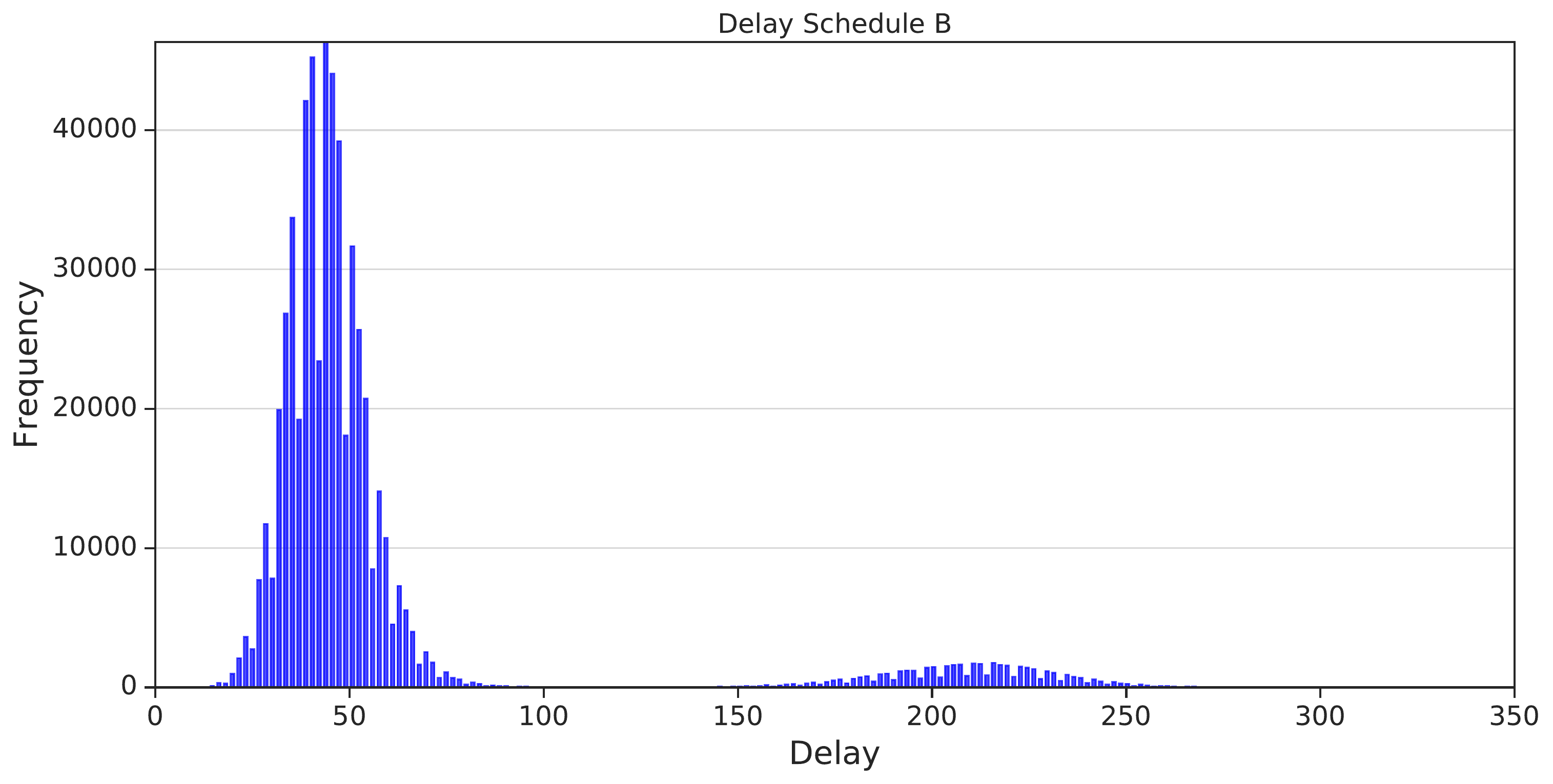}}
\subfigure{\includegraphics[width=.49\textwidth]{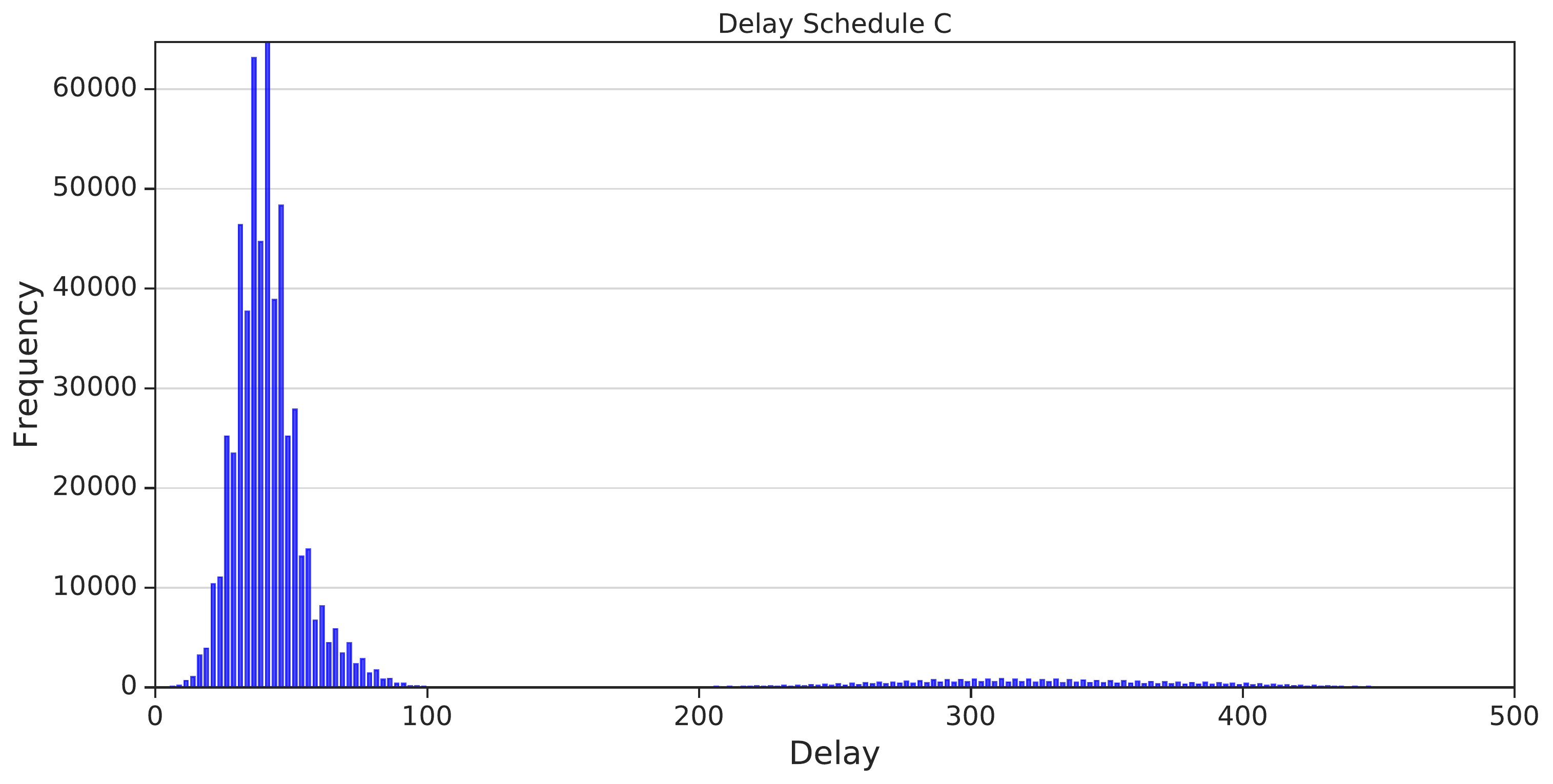}\hfill}
\subfigure{\includegraphics[width=.49\textwidth]{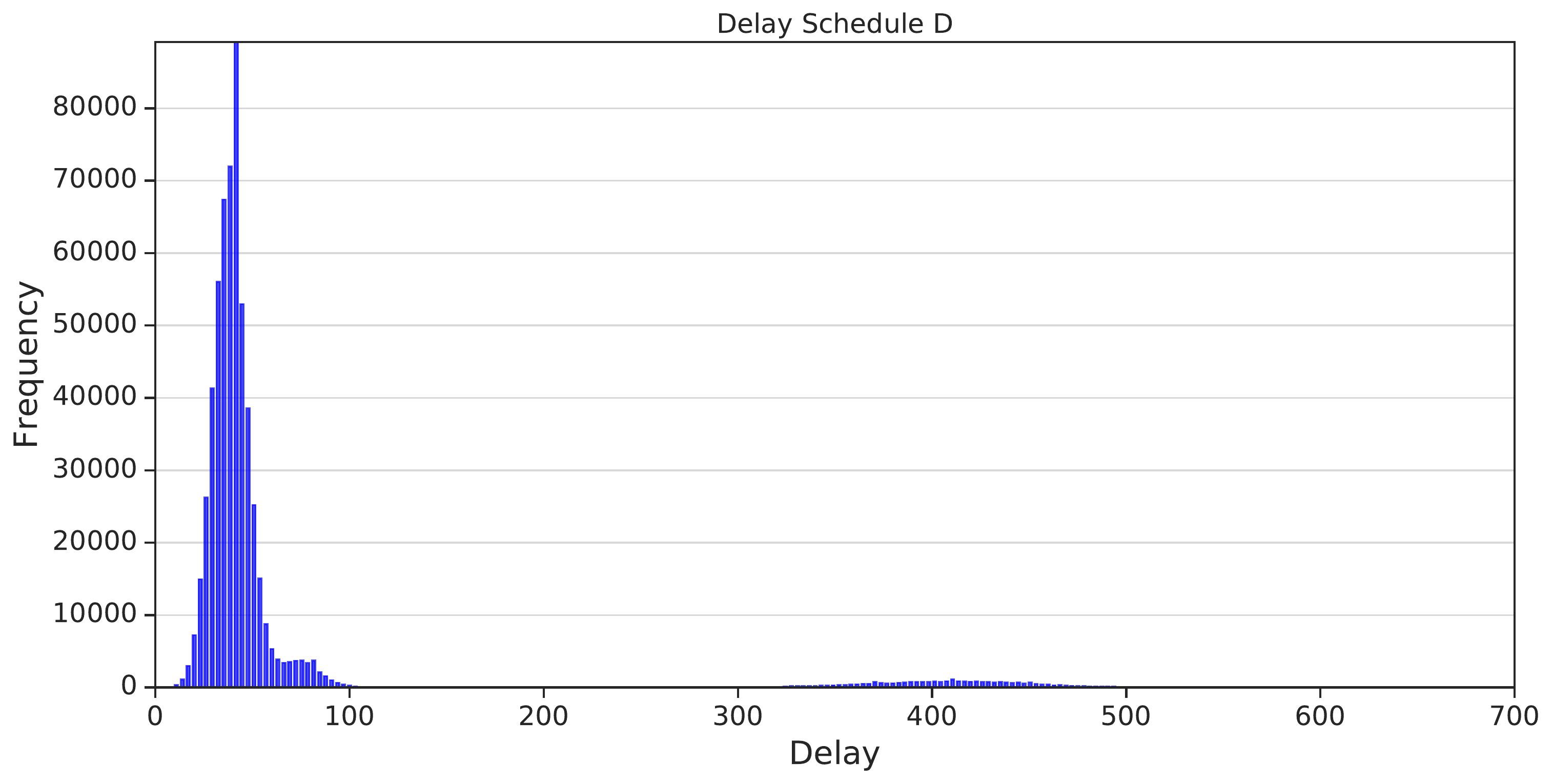}}
\vspace{-0.15in}
\caption{Delay histograms, derived from the delay schedules created for a total of 750 training epochs of CIFAR-10. (A) is the baseline schedule, while (B) (C) and (D) where created with increasing delay variance.}
\label{fig:delay}
\end{figure*}

We used \cref{alg:simgenmaster} to generate four schedules with different statistical properties by varying the number of workers simulated and the sampled wait period distributions as follows:
For schedule $A$ we used $n=10$ workers and Poisson distribution $P$. For the other schedules ($B$, $C$, and $D$) we simulated $75$ workers and used for each schedule a weighted mixture of two Poisson distributions:
$P$ with probability $0.92$ and $150P$ with probability $0.08$ for schedule $B$. $P$ with probability $0.935$ and $240P$ with probability $0.065$ for schedule $C$. And $P$ with probability $0.95$ and $330P$ with probability $0.05$ for schedule $D$. In all schedules we used parameter $4.06$ for the Poisson distribution $P$. Finally, in \cref{alg:simgen}, we scaled the second wait by $0.2$ to reflect the relatively longer time required for gradient computation (compared to the time required for updating the state). The four delay schedules are illustrated in \cref{fig:delay}.

\subsubsection{Meta parameters and further results}
The baseline for the learning rate schedule we used is the one chosen for synchronous optimization of the same data set and model. It starts at a constant rate of $0.05$ and scaled down by $0.1$ at three (fixed upfront) occasions. We used the accuracy achieved at those points ($0.93, 0.98, 0.99$) to define a baseline learning rate schedule $\eta_0$ that behaves the same, but the accuracy of the first rate change being a meta-parameter $R$. We used values from $\{0.8, 0.84, 0.88, 0.93$, $0.96\}$ for $R$.
In addition, the baseline learning rate is further scaled by a meta parameter $K$ (The learning rate multiplier $\ifrac{\eta}{\eta_0}$). Values from $\{0.01, 0.02, 0.05, 0.2, 0.5, 1.0, 2.0\}$ were explored for $K$.
Finally, and only for \algo, we explored values of the threshold $\ifrac{\epsilon}{2\beta}$ in \cref{alg:sgd-with-delays}. 
We aligned changes in the threshold together with the changes in the learning rate (effectively reducing the target accuracy approximation $\epsilon$), and explored thresholds of the form $A\sqrt{\eta_0}$ for values of $A$ in $\{1, 3, 6, 9, 12\}$.

All in all, for every constellation of the meta parameters $R$, $K$, and $A$, we run \cref{alg:simgalg} for SGD and for \algo. The best performing constellation of SGD is compared with the best performing constellation of \algo. \cref{fig:algs} compares the performance trajectory for each of the four delay schedules generated by \cref{alg:simgenmaster}. \cref{fig:algstop} is a comparison of the top three performing constellations of each of the optimization algorithms, further illustrating the robustness and superiority of \algo over SGD.  \cref{table:results} and \cref{fig:algseval2} details the eventual evaluation set performance and trajectory (respectively), for each optimization algorithm and each delay schedule, demonstrating the improved generalization for \algo in all delay schedules.
\begin{table}[h!]
\centering
\vspace{0.1in}
\begin{tabular}{ |p{0.3cm}||p{1.0cm}|p{1.0cm}||p{1.0cm}|p{1.0cm}|}
 \hline
    & \multicolumn{2}{|c||}{Test} 
    & \multicolumn{2}{|c|}{Train} \\
 \hline
  & Picky SGD & SGD & Picky SGD & SGD \\
  \hline
  \hline
 A& 92.98 & 92.56 & 99.87 & 99.82\\
 \hline
 B& 91.82 & 89.64 & 99.62 & 99.25\\
 \hline
 C& 92.12 & 90.09 & 99.66 & 99.23\\
 \hline
 D& 91.80 & 90.15 & 99.61 & 99.22\\
 \hline
\end{tabular}
\label{table:acc_results}
\caption{Eventual (top-1) accuracy of the experiments in \cref{fig:algs} for train and evaluation data sets.}
\end{table}

Finally, we compared the performance of \algo to that of SGD for an alternative learning rate schedule - cosine decay.\footnote{Rather prevalent, although not the one achieving state of the art performance for SGD.}
Using the learning rate decay duration (in epochs) as a meta parameter\footnote{Replacing the $R$ meta parameter of the piece-wise constant learning rate schedule.} with values ranging over \{$150, 180, 240, 300, 360$\}. The training accuracy trajectories of the top 3 meta-parameters constellations of \algo and SGD for the four delay schedules are illustrated in \cref{fig:algseval_cd}, showing an even more pronounced performance gap in favor of \algo, mainly regarding the time to reach the 0.99 accuracy mark and robustness.    


\begin{figure*}[ht]
\centering
\subfigure{\includegraphics[width=.49\textwidth]{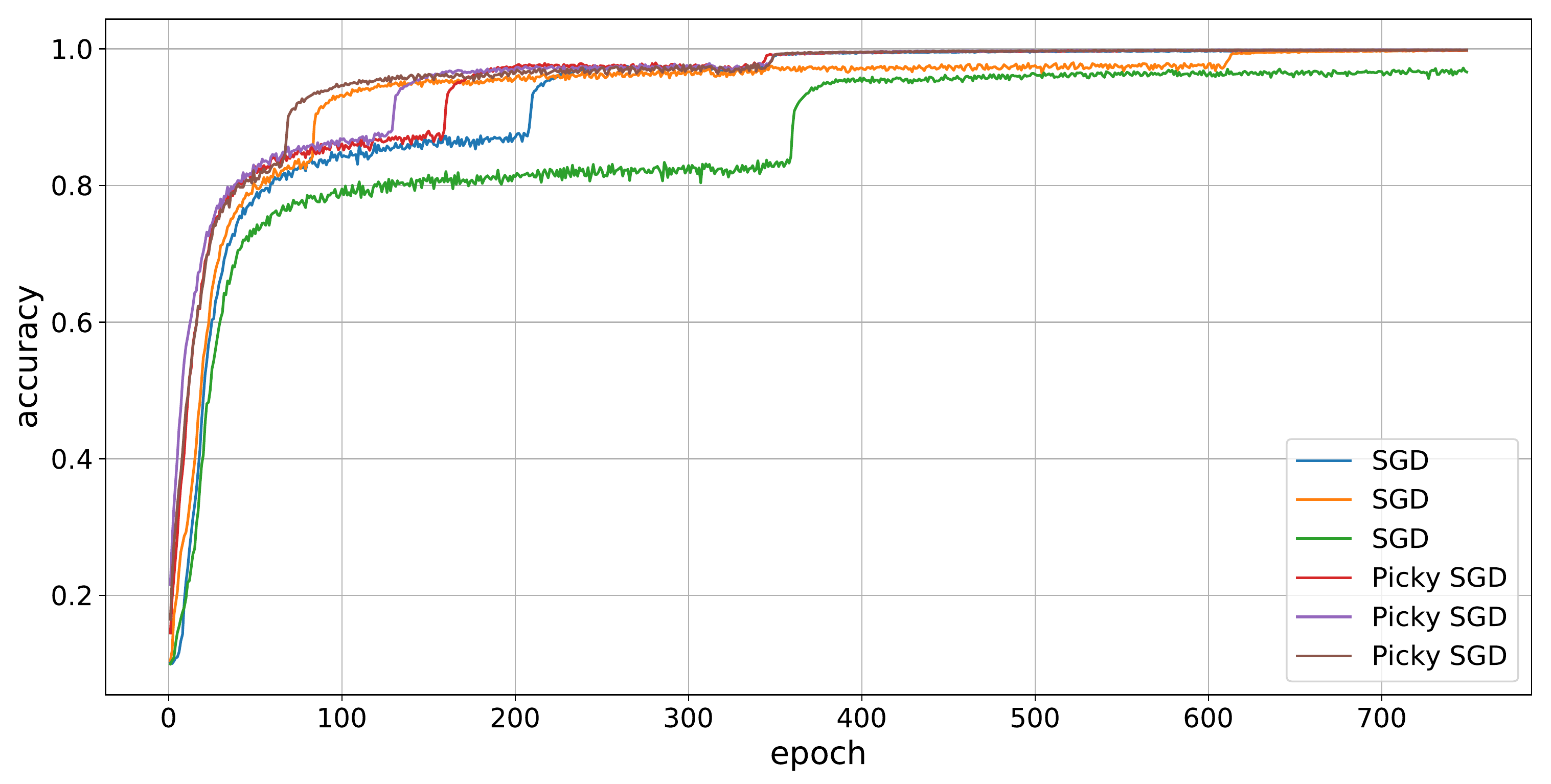}}
\subfigure{\includegraphics[width=.49\textwidth]{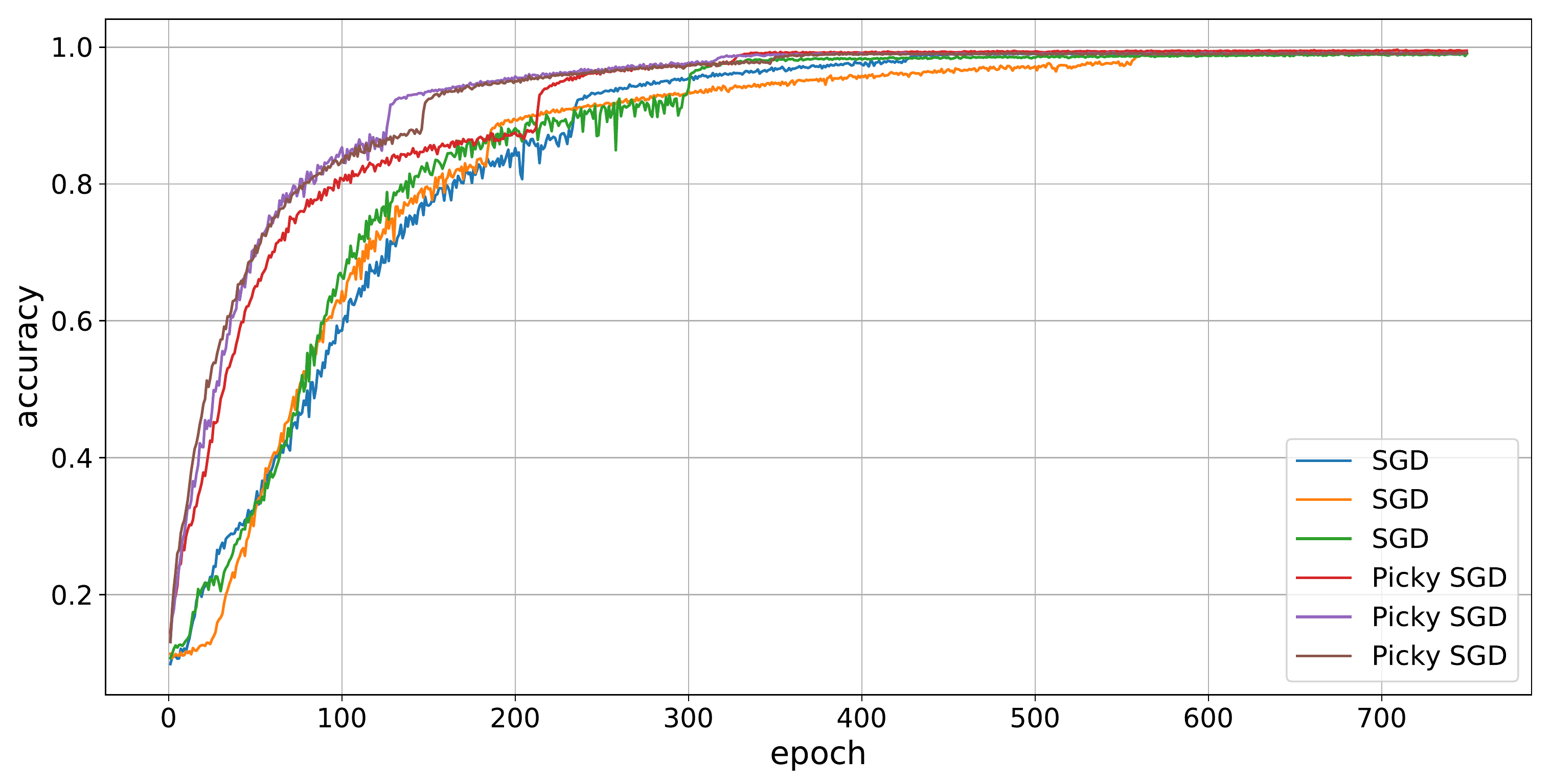}}
\subfigure{\includegraphics[width=.49\textwidth]{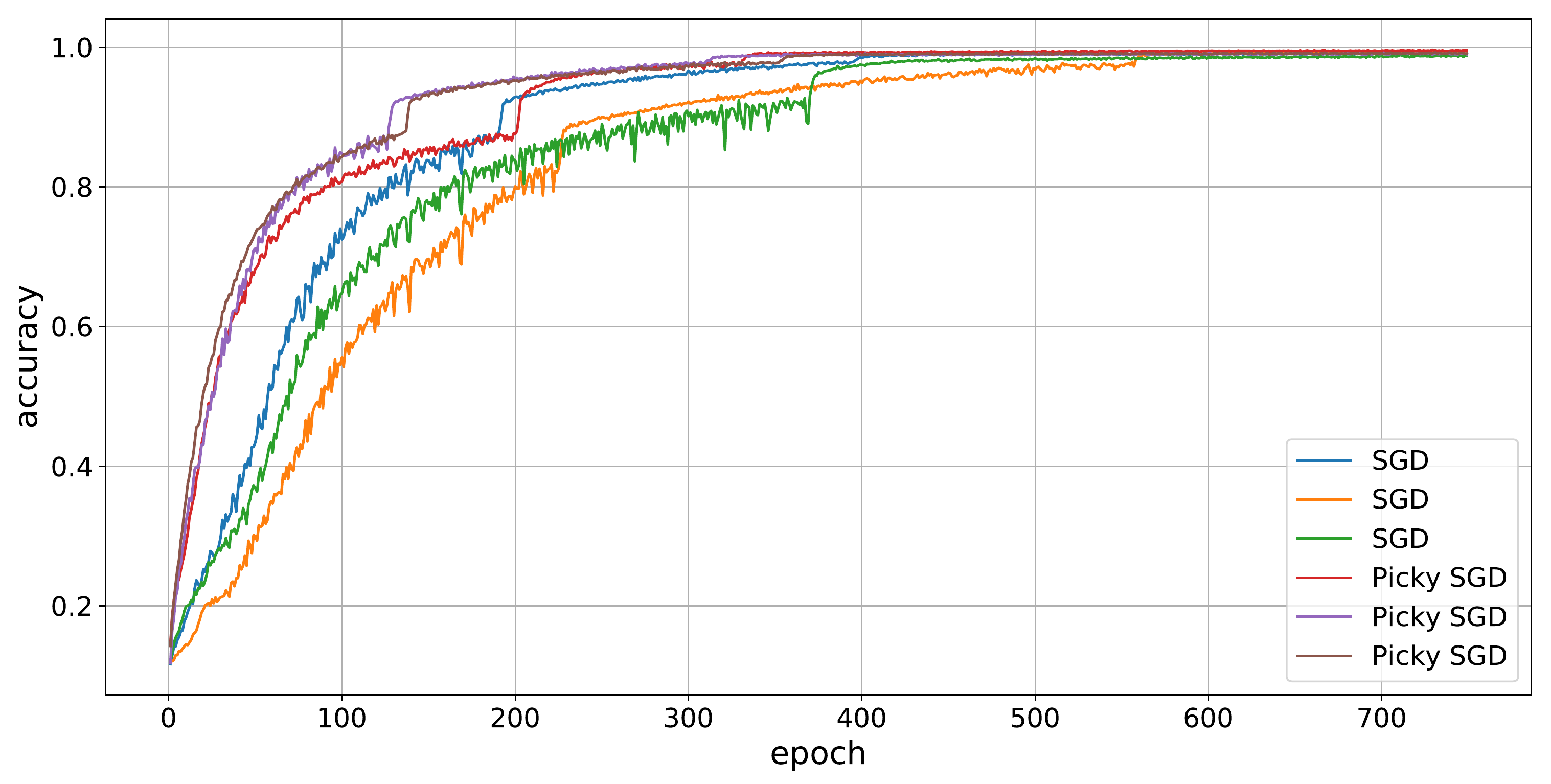}}
\subfigure{\includegraphics[width=.49\textwidth]{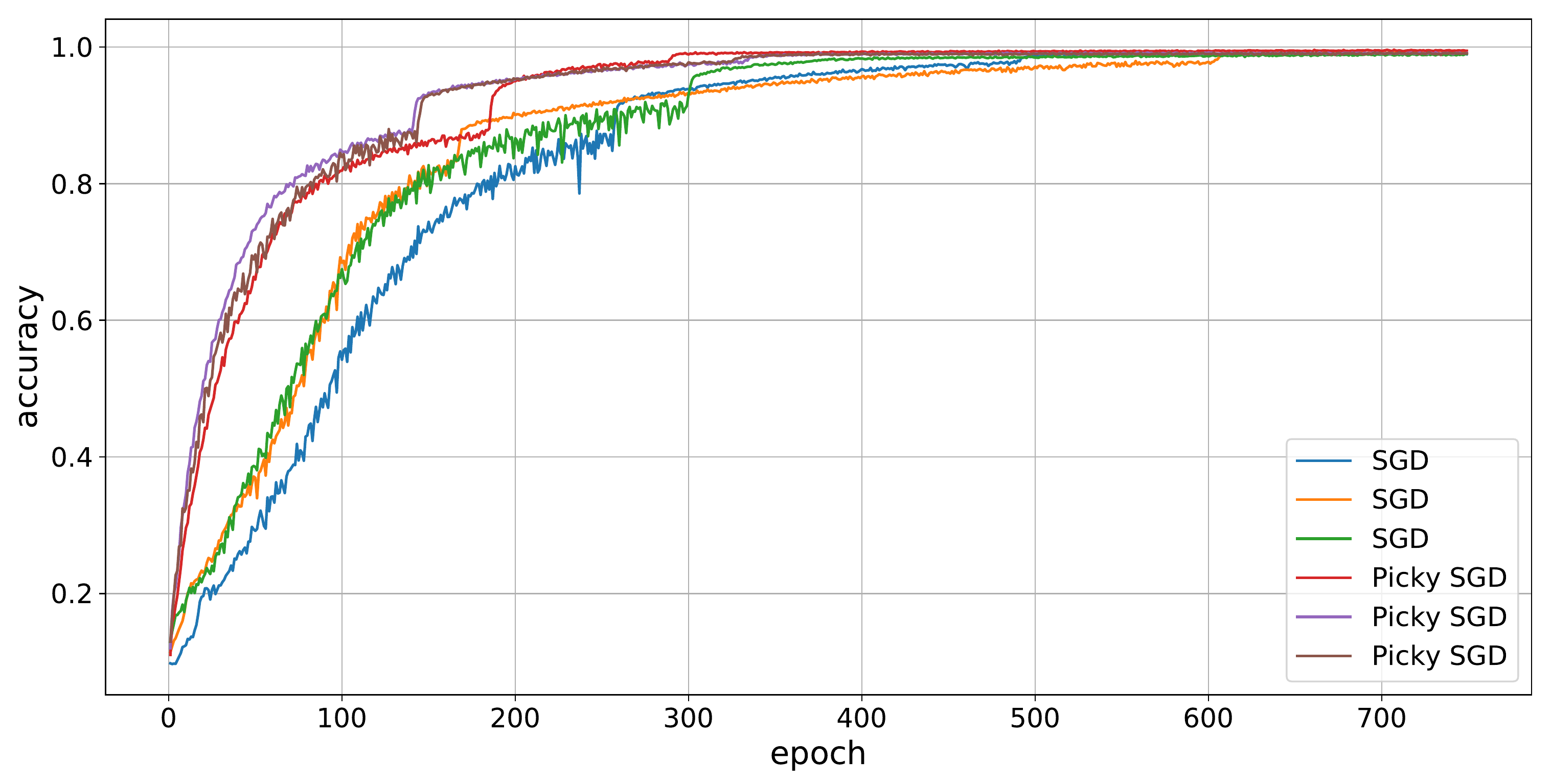}}
\vspace{-0.15in}
\caption{Training accuracy trajectory of top three configurations of \algo and SGD for the four delay schedules of \cref{fig:delay}, respectively. \algo is seen to perform better and more robustly than SGD across different hyperparameter configurations.}
\label{fig:algstop}
\end{figure*}

\begin{figure*}[ht]
\centering
\subfigure{\includegraphics[width=.49\textwidth]{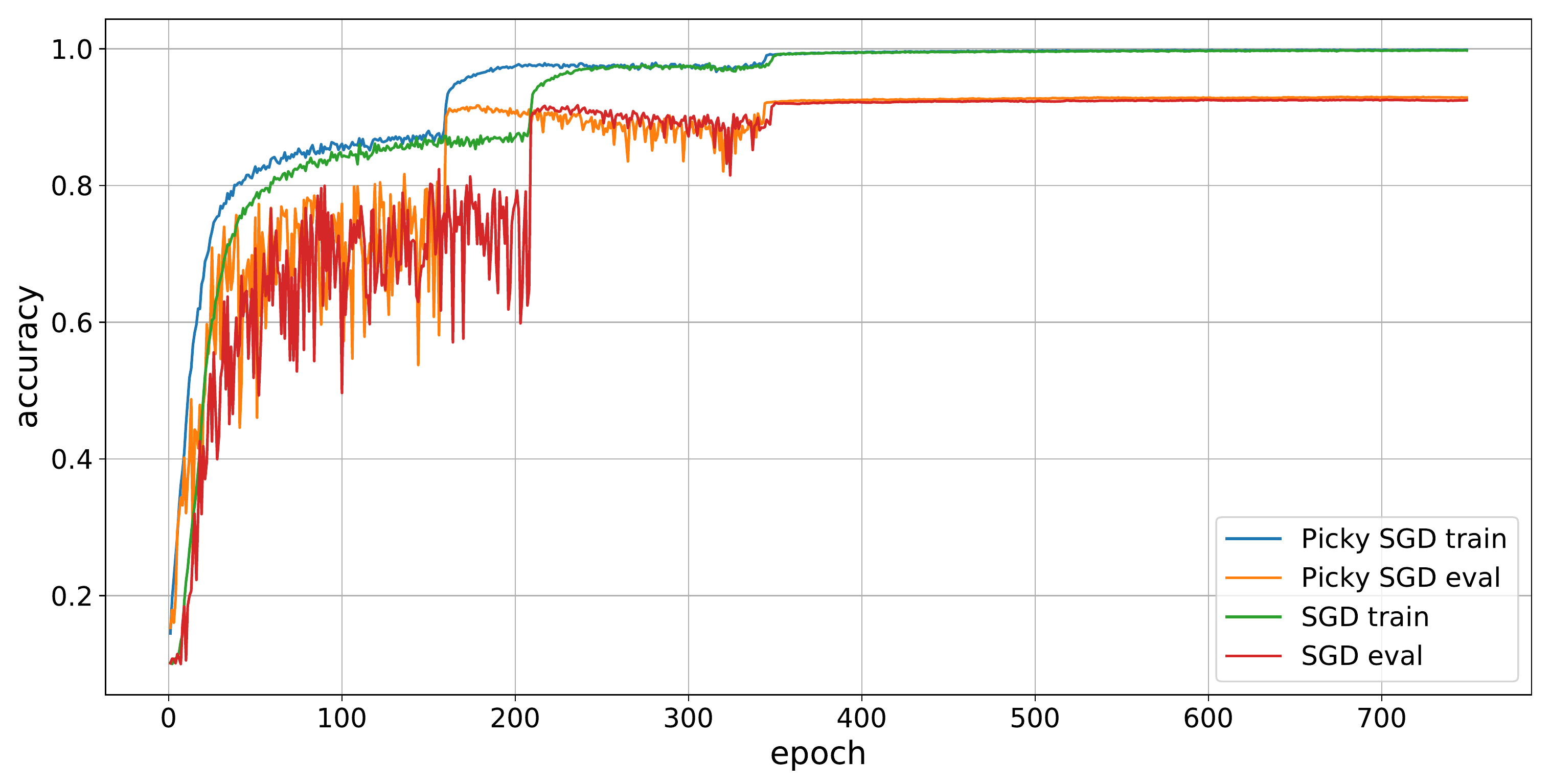}}
\subfigure{\includegraphics[width=.49\textwidth]{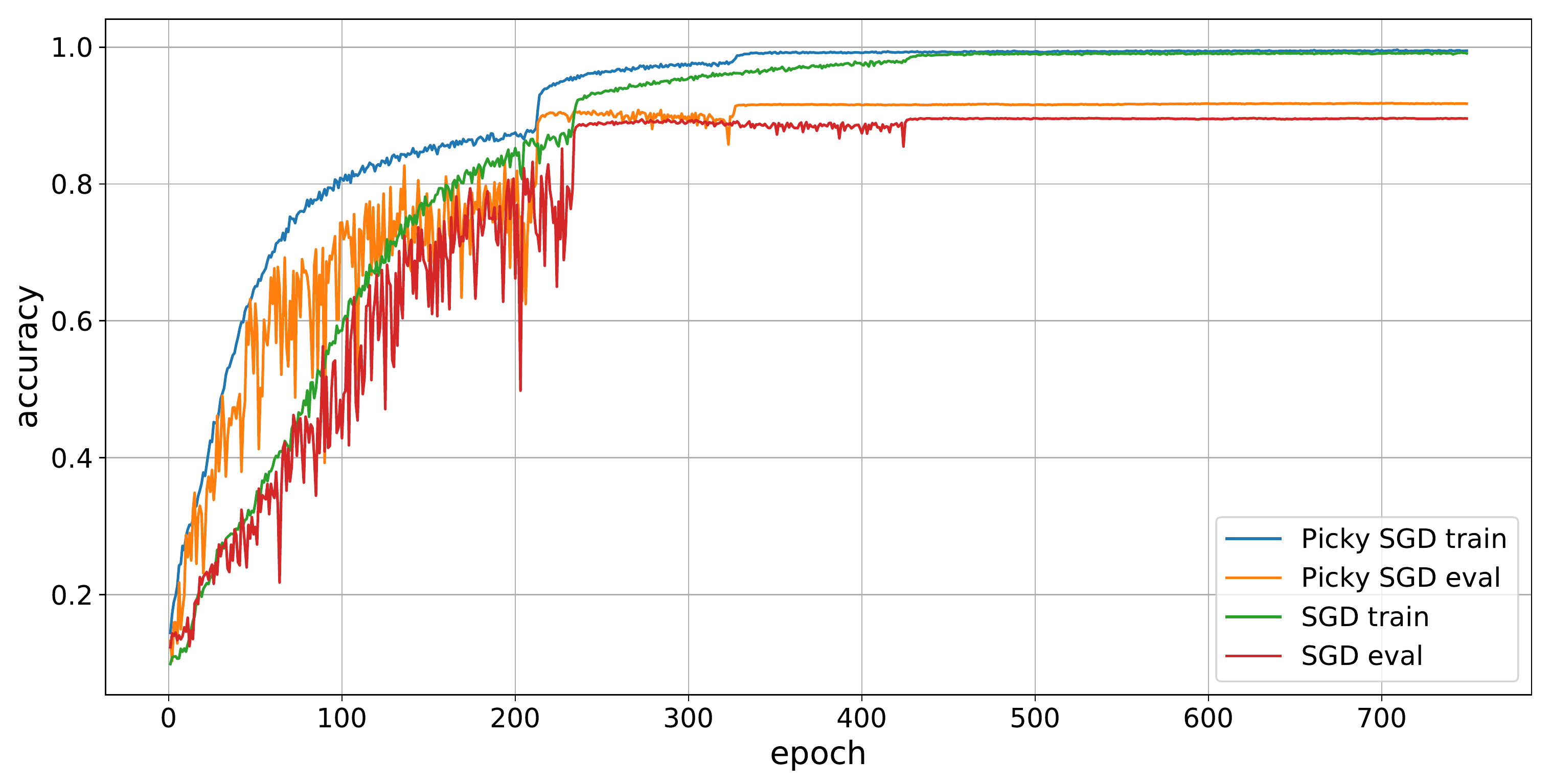}}
\subfigure{\includegraphics[width=.49\textwidth]{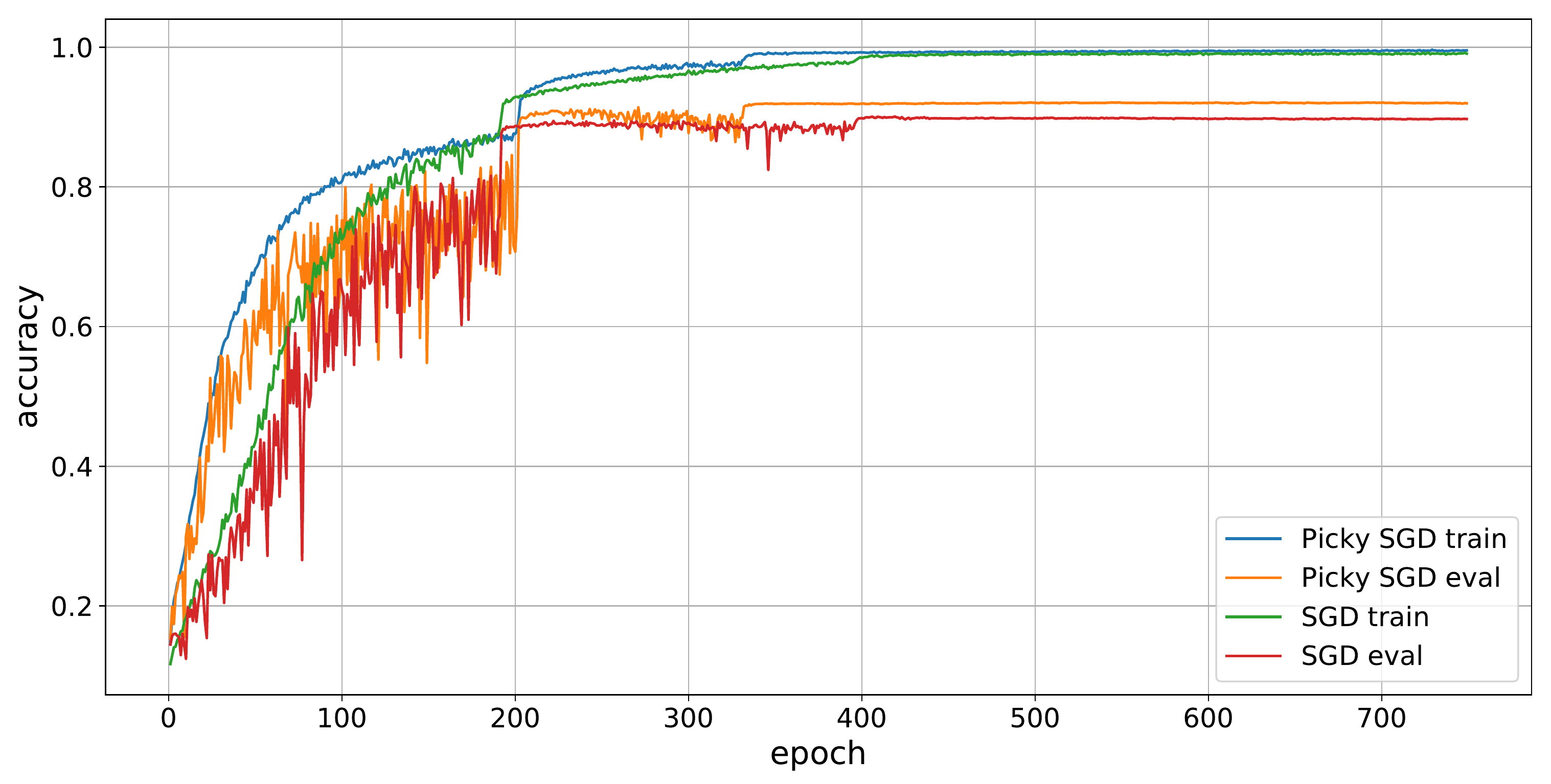}}
\subfigure{\includegraphics[width=.49\textwidth]{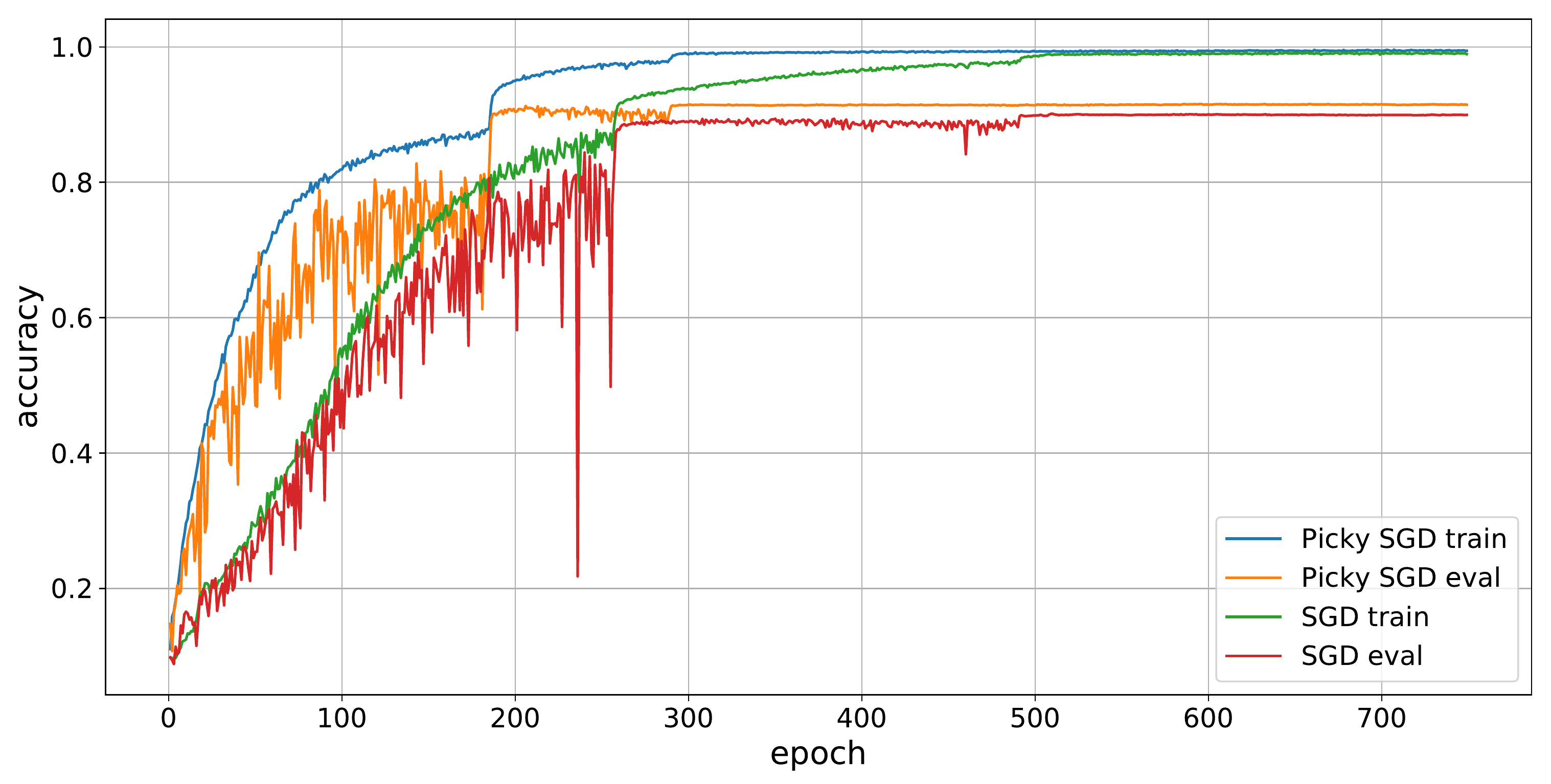}}
\vspace{-0.15in}
\caption{Train and test accuracy trajectories of \algo and SGD for the four delay schedules of \cref{fig:delay}, respectively. 
\algo is seen to outperform SGD in terms of the final test accuracy and the time it takes to achieve it, in all four scenarios.
\label{fig:algseval2}
}
\end{figure*}
\begin{figure*}
\centering
\subfigure{\includegraphics[width=.49\textwidth]{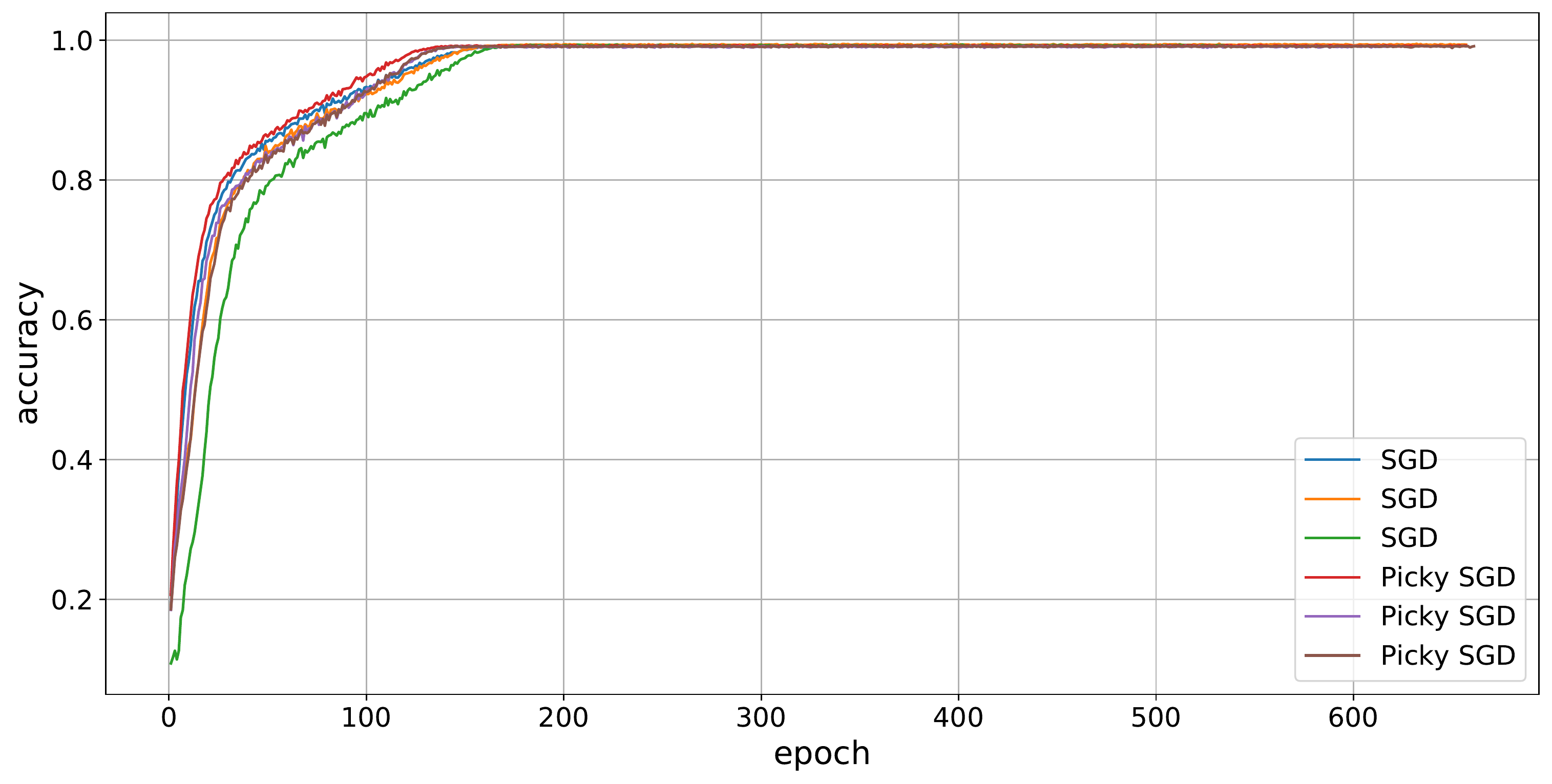}}
\subfigure{\includegraphics[width=.49\textwidth]{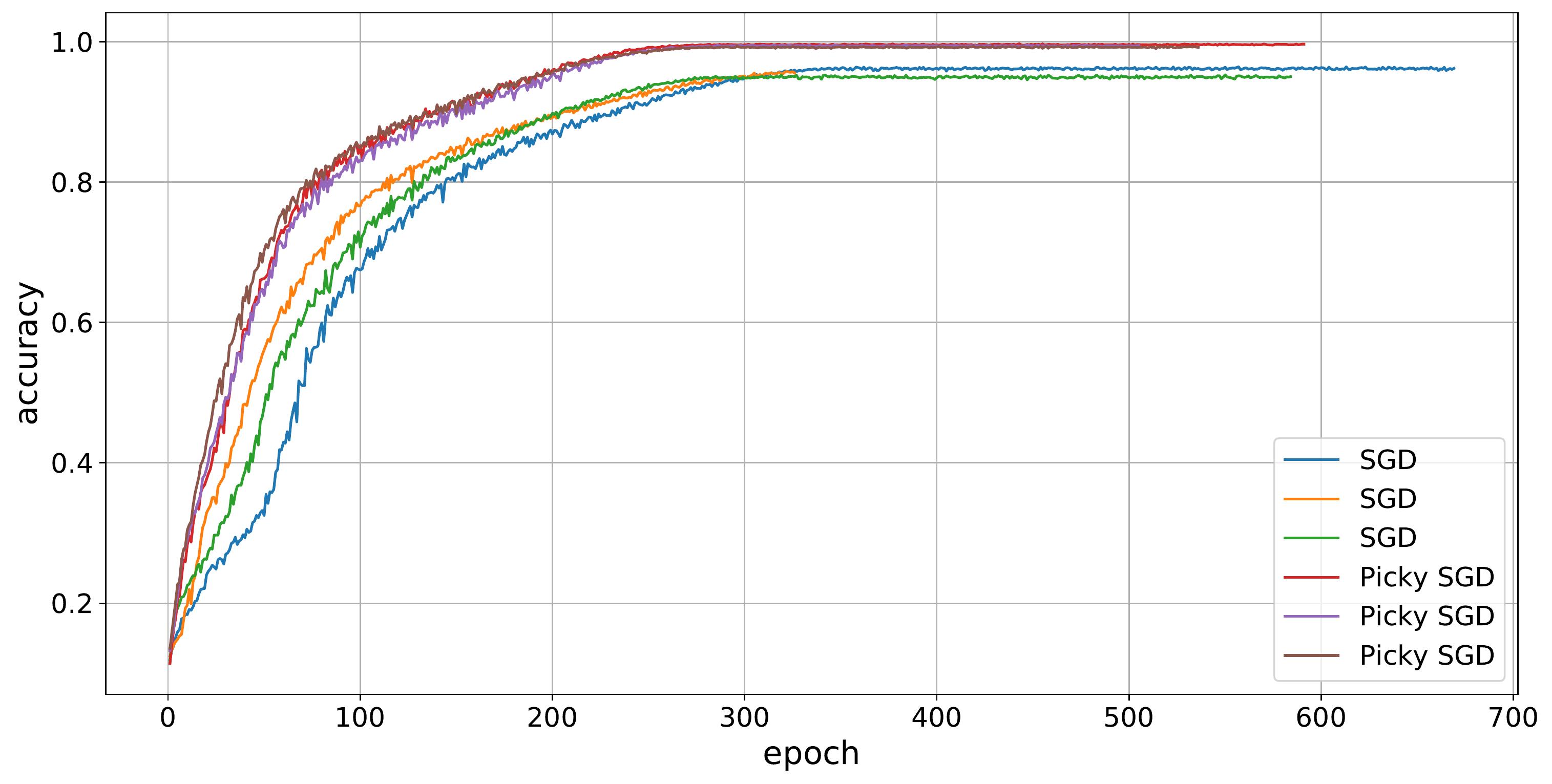}}
\subfigure{\includegraphics[width=.49\textwidth]{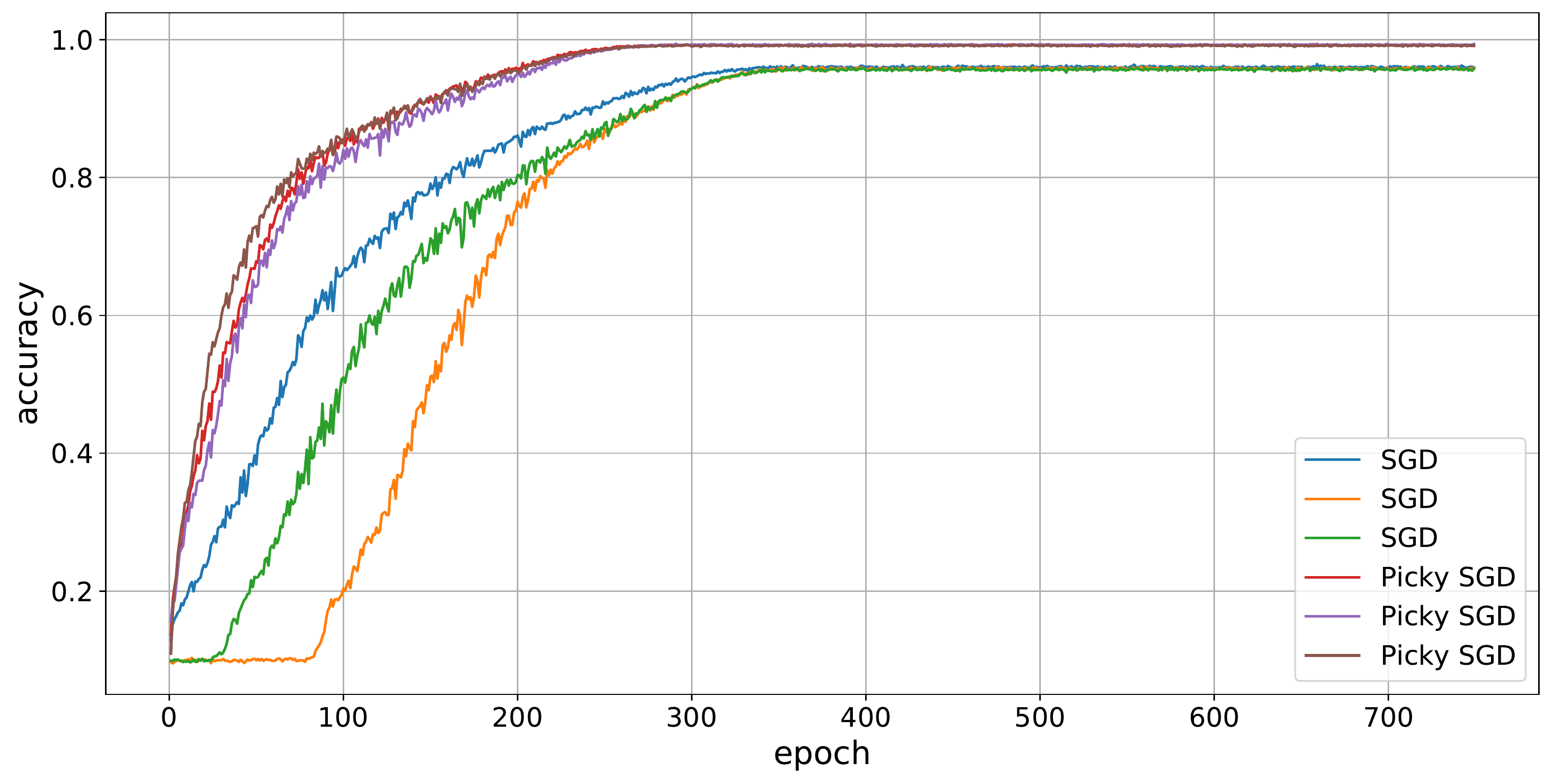}}
\subfigure{\includegraphics[width=.49\textwidth]{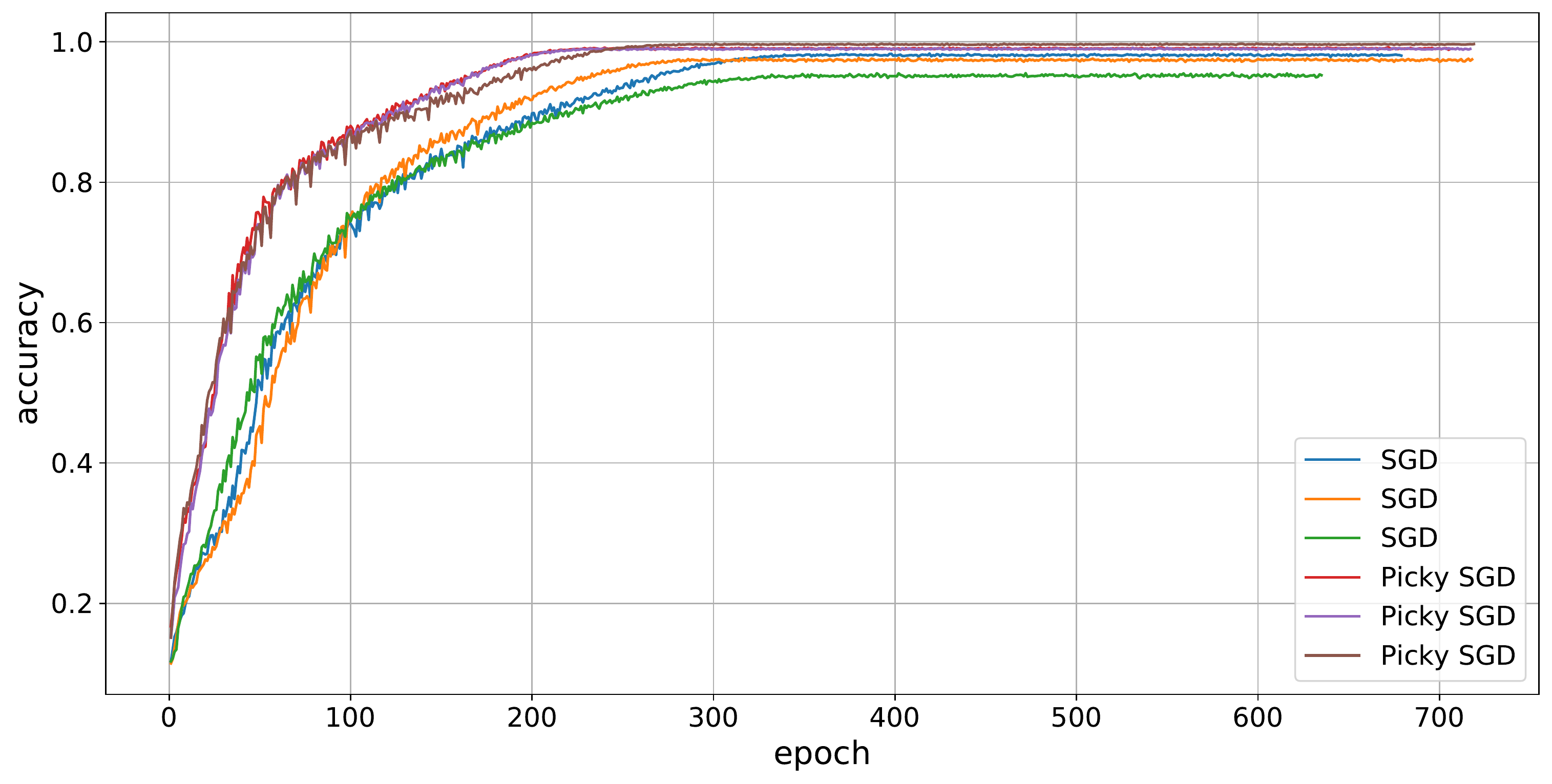}}
\vspace{-0.15in}
\caption{Train accuracy of the top three hyper-parameter configurations for the four delay schedules when trained using cosine-decay learning rate schedule.
\algo is again seen to significantly outperform SGD in terms of the final accuracy and the time it takes to achieve it.
\label{fig:algseval_cd}}
\end{figure*}

\section{Efficient implementation of~\algo}\label{S:efficient}
In this section, we present the implementation of~\algo under a typical multi-worker parameter-server setting. In this setting, the state of the model is stored on a dedicated server called the \emph{parameter server}, while the gradient computation is performed on a set of \emph{worker} machines. At each iteration, the worker queries the parameter server for its current state, computes the gradient then sends the update to the back to server. This architecture allows an efficient use of asymmetrical and computational resources and machines under varying work loads, such as the ones commonly available in large-scale cloud platforms. Note that these systems are particularly amenable to large variations in the computational time of the gradient.

Under the setting described above, the straightforward implementations of \algo are inefficient, where either the parameter server is required to keep a large portion of the history of the computation (if \cref{ln:test-update} of the algorithm is executed by the parameter server) or the workers need to query the parameter server twice during each gradient computation (in case the \cref{ln:test-update} is executed by the workers).

The overhead described above can be eliminated by executing \cref{ln:compare-dist} at the worker side and observing that after sending the appropriate update to the parameter server, the worker has all the information it needs in order to compute the next iterate without an additional query to the parameter server. See \cref{alg:sgd-with-delays-worker} for a pseudo-code describing the worker side of the proposed method (the parameter server implementation proceeds simply by receiving the gradient and updating the parameter state accordingly).

\begin{algorithm}
    \caption{\algo: worker implementation} \label{alg:sgd-with-delays-worker}
    \begin{algorithmic}[1]
        \STATE {\bf input}: 
            learning rate $\eta$, 
            target accuracy $\epsilon$.
        \STATE {\bf query parameter server} $\rightarrow x$.
        \LOOP{}
            \STATE {\bf compute} stochastic gradient $g$  such that $\E[g] = \nabla f(x)$.
            \STATE {\bf query parameter server} $\rightarrow x'$.
            \IF{$\norm{x - x'} \le \ifrac{\epsilon}{(2 \beta)}$}
            \STATE {\bf send update to parameter server} $\leftarrow g$.
            \STATE {\bf set:} $x \leftarrow x' - \eta g$.
            \ELSE 
                \STATE {\bf set:} $x \leftarrow x'$.
            \ENDIF
        \ENDLOOP
    \end{algorithmic}
\end{algorithm}


{Finally, regarding the tuning of the threshold parameter $\ifrac{\epsilon}{(2 \beta)}$, a simple strategy of selecting a good threshold we found to be effective in practice, is to log all distances $\|x-x'\|$ that occur during a typical execution and taking 99th percentile of these distances as the threshold. This ensures robustness to long delays while maintaining near-optimal performance.}

\end{document}